\documentclass[11pt]{amsart}
\usepackage{amssymb,amsmath,amsfonts,amsthm}
\usepackage{graphicx}
\graphicspath{ {./images/} }
\usepackage{mathrsfs}
\usepackage{psfrag}
\usepackage{mathtools}
\usepackage{color}
\usepackage{todonotes}
\usepackage{enumitem}
\usepackage{booktabs,array,longtable}
\usepackage{chngcntr}
\theoremstyle{plain}
\newtheorem{main}{Theorem}

\newtheorem{theorem}{Theorem}[section]
\newtheorem{lemma}[theorem]{Lemma}
\newtheorem{proposition}[theorem]{Proposition}
\newtheorem{corollary}[theorem]{Corollary}

\theoremstyle{remark}
\newtheorem{remark}[theorem]{Remark}
\newtheorem{definition}[theorem]{Definition}

\newcommand\numberthis{\addtocounter{equation}{1}\tag{\theequation}}

\newcommand{\Sing}{\operatorname{Sing}}

           \def\ea{\end{array}}
          \def\ec{\end{center}}
     \def\ed{\end{description}}
        \def\ee{\end{equation}}
       \def\eea{\end{eqnarray}}
     \def\eeaa{\end{eqnarray*}}
 \def\et{\end{thebibliography}}

\def\Sing{{\rm Sing}}

\usepackage{hyperref}
\hypersetup{colorlinks=true, citecolor=green, linkcolor=blue, hypertexnames=false}

\def\cG{{\mathcal G}}
\def\cA{{\mathcal A}}

\def\cC{{\mathcal C}}
\def\cO{{\mathcal O}}

\def\cU{{\mathcal U}}
\def\cV{{\mathcal V}}
\def\cR{{\mathcal R}}

\def\cF{{\mathcal F}}

\def\cN{{\mathcal N}}
\def\cP{{\mathcal P}}

\def\cR{{\mathcal R}}

\def\vep{\varepsilon}

\def\RR{{\mathbb R}}

\def\NN{{\mathbb N}}

\def\sv{{\rm sv}}

\def\CR{\operatorname{CR}}

\def\Sing{\operatorname{Sing}}
\def\Reg{\operatorname{Reg}}
\def\Int{\operatorname{Int}}

\def\xX{{\mathscr{X}}}

\title[Komuro expansivity and periodic orbits counting]{Komuro Expansivity and Periodic Orbit Growth for Multi-Singular Hyperbolic Flows}

\author{M. J. Pacifico, F. Yang, J. Yang and R. Zheng}

\address{Instituto de Matem\'atica, Universidade Federal do Rio de Janeiro, C. P. 68.530, CEP 21.945-970,  Rio de Janeiro, RJ, Brazil.}
\email{pacifico@im.ufrj.br }

\address{Department of Mathematics, Wake Forest University, Winston-Salem, NC, USA.}
\email{yangf@wfu.edu}

\address{Departamento de Geometria, Instituto de Matem\'atica e Estat\'\i stica, Universidade Federal Fluminense, Niter\'oi, Brazil}
\email{jiagangyang@id.uff.br}

\address{Department of Mathematics, Southern University of Science and Technology, Shenzhen, China.}
\email{rszheng0822@gmail.com}

\thanks{Pacifico’s work was partially supported by CAPES-Finance Code 001, CNPq Projeto Universal No. 404943/2023-3, CNPq-Brazil grant 307776/2019-0 and by
	Foundation for Research Support of the State of Rio de Janeiro (FAPERJ) grant CNE
	E-26/202.850/2018(239069). J.\ Yang's work was
	partially supported by CAPES-Finance Code 001, CNPq-Brazil grant 312054/2023-8,
	CNPq-Projeto Universal No. 404943/2023-3, PRONEX and
	MATH-AmSud 220029. F.\ Yang’s work was partially supported by National Science
	Foundation (NSF) grant DMS-2418590.}       

\begin{document}

\begin{abstract}In this paper we prove that every multi-singular hyperbolic set of a $C^1$ flow is Komuro expansive. This is the strongest natural form of expansivity for flows with singularities accumulated by regular orbits, allowing arbitrary orientation-preserving time reparametrizations. We then use this to establish a two-sided asymptotic counting bound of the form $e^{ht}/t$ for periodic orbits, and prove that the normalized orbit measures converge to the unique measure of maximal entropy. We also establish the same results for a $C^1$ open and dense subset of star vector fields. 
\end{abstract}

\maketitle
\tableofcontents
\section{Introduction}

Uniformly hyperbolic flows provide the classical setting in which orbit
separation, invariant measures, and periodic orbits can all be described
with great precision. Expansivity separates distinct trajectories up to a
small displacement along the flow direction. Using such separation, the work of Margulis,
Bowen, and Parry--Pollicott established prime-orbit counting and
equidistribution results for geodesic, Anosov, and Axiom~A flows. Under the
standard transitivity and non-arithmetic hypotheses, the number of prime
periodic orbits of length at most $t$ is asymptotic to $e^{ht}/(ht)$, and
the associated periodic-orbit measures converge to the measure of maximal
entropy \cite{Margulis,B72,PP}. These results form a basic
model for the relation between topological entropy and the distribution of
periodic trajectories.

The presence of singularities changes this picture substantially. The
Lorenz equations \cite{Lo63} and their geometric models \cite{GW,ABS}
exhibit robust chaotic behavior in which regular trajectories accumulate
on an equilibrium. Such systems are not uniformly hyperbolic, and the
Bowen--Walters notion of expansivity is too strong: whenever a singularity
is accumulated by regular points, Bowen--Walters expansivity \cite{BW72} fails.
Consequently, among the standard notions allowing arbitrary orbitwise
closeness, Komuro expansivity is the strongest form of expansivity one can
expect in this setting. Komuro introduced this notion and proved it for
geometric Lorenz attractors \cite{Ko,Ko1}. The time
reparametrization is required to be an orientation-preserving
homeomorphism of $\mathbb R$, so that both points traverse their full
orbits. In particular, Komuro expansivity genuinely separates distinct
periodic orbits within one period, and is therefore the natural singular analogue of the orbit-separation mechanism that underlies the aforementioned classical periodic-orbit
theory of Margulis, Bowen, and Parry--Pollicott.

A major structural advance was made by Morales, Pac\'ifico, and Pujals
\cite{MPP99, MPP04}, who proved that every $C^1$ robustly transitive set with
singularities for a flow on a closed three-manifold is a partially
hyperbolic attractor or repeller with volume-expanding central direction.
This led to the theory of singular hyperbolicity, which provides the
appropriate robust replacement for uniform hyperbolicity in dimension
three and includes the geometric Lorenz attractor. Ara\'ujo, Pac\'ifico,
Pujals, and Viana subsequently proved that singular-hyperbolic attractors
are expansive, in addition to establishing their basic statistical
properties \cite{APPV}. Thus, in dimension three, the geometric and
ergodic theory of Lorenz-like attractors can be developed in close analogy
with the uniformly hyperbolic theory. For a systematic account of three-dimensional flows and
singular-hyperbolic attractors, we refer to \cite{ArPa10}.

In higher dimensions, singular hyperbolicity (or sectional hyperbolicity proposed in \cite{MM}) is no longer flexible enough:
a robust chain recurrence class may contain singularities of different
indices. Such an example was given by da Luz \cite{daLuz}, and later, Bonatti and da Luz introduced multi-singular hyperbolicity in their seminal paper \cite{BdL} to
describe this phenomenon and showed that, on an open and dense subset of
star flows, the chain recurrent set admits a finite filtration into
multi-singular hyperbolic pieces. % We use the equivalent formulation developed in \cite{CLYZ,PYY25}. 

Recently, the first three authors of this article established expansivity for
sectional-hyperbolic flows \cite{PYY21}, and later
developed the thermodynamic formalism, including existence and uniqueness
of equilibrium states, for sectional-hyperbolic attractors in arbitrary dimension \cite{PYY25}.
Wen and Wen proved that multi-singular hyperbolic sets are rescaling
expansive \cite{WW}. More recently, expansivity and entropy expansiveness
were extended to the multi-singular hyperbolic setting
\cite{PYY25a} and was used to establish the Thermodynamical formalism for star flows. Full Komuro expansivity, however, has remained open beyond
the previously known Lorenz and three-dimensional singular-hyperbolic
settings.

%, Wen and Wen proved that
%multi-singular hyperbolic sets are rescaling expansive \cite{WW}, where
%the allowed spatial error decreases proportionally to the speed of the
%flow near a singularity. More recently, expansivity and
%entropy expansiveness have been established for all sectional- and
%multi-singular hyperbolic flows in arbitrary dimension \cite{PYY21,PYY25a}. Full Komuro
%expansivity, however, has remained open beyond the previously known
%Lorenz and three-dimensional singular-hyperbolic settings.

The first purpose of this paper is to close this gap. We prove that every
multi-singular hyperbolic set is robustly Komuro expansive without any
transitivity assumption and with an expansivity constant that is uniform
under small $C^1$ perturbations. 

Our second purpose is to develop periodic orbit theory for
multi-singular hyperbolic flows. For a positive-entropy multi-singular
chain recurrence class in a $C^1$ open and dense subset of all vector fields,
we prove two-sided Margulis-type bounds
\[
\#\Pi(t,\Delta)= C^{\pm 1}\frac{e^{ht}}t,
\qquad
\#\Pi(t)= C^{\pm 1}\ \frac{e^{ht}}t,
\]
and show that both the fixed-window and cumulative periodic-orbit averages
converge to the unique measure of maximal entropy. Previously, only the lower bound for the exponential growth rate of the form $$\limsup_{t\to\infty}\frac{1}{t}\log \#\Pi(t) \ge h_{top}(X)$$ was proven in \cite{WYZ} under a generic assumption. We also obtain open and
dense versions of Komuro expansivity, counting, and equidistribution for
positive-entropy star flows. The proof combines the geometric control
provided by multi-singular hyperbolicity \cite{PYY23, PYY25, PYY25a}, thermodynamic estimates \cite{PYY22, PYYY}, 
Liao's theory \cite{Li74, GY,LGW, WW,Gan}, and a phase-separation argument that converts
Komuro expansivity into the factor $1/t$ in the upper counting bound. This provides a sharper estimate when compared to the corresponding results in \cite{BCFT}.

\subsection{Background}

Denote by $\xX^1(M)$ the set of all $C^1$ vector fields on a given compact Riemannian manifold without boundary, and let $X\in \xX^1(M)$ be a $C^1$ vector field on a closed Riemannian manifold $M$. Denote by $(f_t)_{t\in\mathbb{R}}$ the flow generated by $X$. 
We denote by $\Sing(X)$ the set of all singularities of $X$, i.e 
\[\Sing(X)=\{x\in M: X(x)=0\}.\] 
A singularity is called hyperbolic if it is a hyperbolic fixed point of the time-$t$ map $f_t$ for any $t\ne 0$. The following classification is motivated by \cite{MPP99, MPP04, SGW}.

\begin{definition}[Lorenz-like singularity]
	Let $\sigma\in\Sing(X)$ be a hyperbolic singularity with the hyperbolic splitting $T_{\sigma}M=E^s\oplus E^u$. 
	\begin{itemize}
		\item The singularity is said to be {\em Lorenz-like} if its stable subspace $E^s$ splits into a dominated splitting $E^s=E^{ss}\oplus E^c$ such that $\dim E^c=1$, and moreover, letting $\lambda^{u}$ be the smallest Lyapunov exponent along $E^u$, and $\lambda^c$ be the Lyapunov exponent along $E^c$, one has 
	\[\lambda^c<0<-\lambda^c<\lambda^u.\]
	We denote $\sv(\sigma)=\lambda^c+\lambda^u>0$, which will be called the {\em saddle value} of $\sigma$.
	\item The singularity is said to be {\em reversed Lorenz-like} if it is Lorenz-like for the reversed vector field $-X$. More precisely, its unstable subspace $E^u$ splits into a dominated splitting $E^u=E^{c}\oplus E^{uu}$ such that $\dim E^c=1$ and \[\lambda^s<-\lambda^c<0<\lambda^c,\]
	where $\lambda^{s}$ is the largest Lyapunov exponent along $E^s$ and $\lambda^c$ is the Lyapunov exponent along $E^c$. 
	In this case, the saddle value of $\sigma$ is defined to be $\sv(\sigma)=\lambda^s+\lambda^c<0$. 
	\end{itemize}
\end{definition}
%A singularity $\sigma\in\Lambda$ is called {\em active}, if there exist regular orbits in $\Lambda$ that gets arbitrarily close to $\sigma$ in both forward and backward iteration. More notably, Theorem \ref{m.counting} will assume that 

For each $x\in M\setminus\Sing(X)$, we denote 
\[N(x)=\{v\in T_xM: v\perp X(x)\},\]
which is called the {\em normal space} at $x$. 
Let $\Lambda$ be any compact invariant set of the flow $(f_t)_{t\in\mathbb{R}}$ such that $\Lambda\setminus\Sing(X)\neq\emptyset$. We define the {\em normal bundle} over $\Lambda$ as 
\[N_{\Lambda}=\bigcup_{x\in\Lambda\setminus\Sing(X)}N(x).\]
Write $(\psi_t)_{t\in\mathbb{R}}$ for the linear Poincar\'e flow, that is, the orthogonal projection of the tangent flow $(Df_t)_{t\in\RR}$ to $N_M$, the normal bundle over $M$. In addition, we defined the {\em scaled linear Poincar\'e flow} as 
$$
\psi^*_t(v)=\frac{|X(x)|}{|X(f_t(x))|}\psi_t(v)=\frac{\psi_t(v)}{\|Df_t|_{\langle X(x)\rangle}\|}.
$$

Now suppose there is an invariant splitting $N_{\Lambda}=E\oplus F$ with respect to the linear Poincar\'e flow $(\psi_t)$, such that for some constants $\eta, T>0$ it holds
\[\|\psi_t|_{E(x)}\|\cdot\|\psi_{-t}|_{F(f_t(x))}\|<e^{-\eta t},\quad \forall x\in \Lambda\setminus\Sing(X), t\ge T,\]
then the splitting $N_{\Lambda}=E\oplus F$ is said to be {\em dominated} with respect to the linear Poincar\'e flow. One also says that $N_{\Lambda}=E\oplus F$ is an {\em $(\eta,T)$-dominated splitting}. When the dimension of $E(x)$ is a constant over $\Lambda$ (which will be assumed in general), it will be called the {\em index} of the splitting over $\Lambda$.

We now introduce {\em multi-singular hyperbolicity}, first defined by Bonatti and da Luz in \cite{BdL}. The version cited here is from \cite{CLYZ} (see also \cite{PYY25}). It is known that, under a mild assumption, these two definitions are equivalent. See Theorems D and E in \cite{CLYZ}.

\begin{definition}\label{def.multi-sing-hyp}
	Let $\Lambda$ be a compact invariant set of the flow  $f_t$ such that $\Lambda\setminus\Sing(X)\neq\emptyset$. We say that $\Lambda$ is {\em multi-singular hyperbolic} if the following properties hold:
	\begin{enumerate}
		\item there is a dominated splitting $N_{\Lambda}=N^{cs}\oplus N^{cu}$ with respect to the linear Poincar\'e flow, with constant index $\dim N^{cs}$.
		\item for each compact isolating neighborhood $V$ of $\Lambda\cap\Sing(X)$ small enough, there are constants $\eta, T>0$ such that 
		\begin{equation}\label{eq.muti-sing-hyp}
			\|\psi_t|_{N^{cs}(x)}\|<e^{-\eta t}, \quad \|\psi_{-t}|_{N^{cu}(f_t(x))}\|<e^{-\eta t},
		\end{equation}
		whenever $x,f_t(x)\in\Lambda\setminus V$ and $t\ge T$;
		\item each singularity $\sigma\in\Lambda\cap\Sing(X)$ is either Lorenz-like or reversed Lorenz-like; moreover, 
		\begin{itemize}
			\item if $\sigma$ is Lorenz-like, then $\dim E^{ss}_{\sigma}=\dim N^{cs}$ and $W^{ss}(\sigma)\cap\Lambda=\{\sigma\}$, where $W^{ss}(\sigma)$ is the strong stable manifold of $\sigma$ that is tangent to $E^{ss}_{\sigma}$;
			\item if $\sigma$ is reversed Lorenz-like, then $\dim E^{uu}_{\sigma}=\dim N^{cu}$ and $W^{uu}(\sigma)\cap\Lambda=\{\sigma\}$, where $W^{uu}(\sigma)$ is the strong unstable manifold of $\sigma$ that is tangent to $E^{uu}_{\sigma}$.
		\end{itemize}
	\end{enumerate}
\end{definition}
%\begin{remark}\label{rmk.multi-sing-hyp}
%	The neighborhood $V$ can be assumed to be a compact $\vep$-neighborhood of $\Sing(X)\cap\Lambda$, with $\vep>0$ arbitrarily small. 
%\end{remark}
\begin{remark}\label{rmk.msh-nonsingular}
	The inequalities in \eqref{eq.muti-sing-hyp} imply that any nonempty compact invariant subset $\Lambda_0\subset \Lambda$ containing no singularities is a uniformly hyperbolic set. 
\end{remark}

%Suppose that $\Lambda$ is singular hyperbolic. % We consider firstly the case of singular hyperbolic sets
%Then all singularities in $\Lambda$ is Lorenz-like. Precisely, for any $\sigma\in\Lambda\cap\Sing(X)$, there is a three-way dominated splitting $T_{\sigma}M=E^{ss}\oplus E^c\oplus E^u$ such that $\dim E^c_{\sigma}=1$ and $E^{ss}_{\sigma}\oplus E^c_{\sigma}$ corresponds to the stable subspace of $T_{\sigma}M$, $E^u_{\sigma}$ is the unstable subspace of $\sigma$; moreover, its saddle value $\sv(\sigma)>0$. 

{
	
Since several inequivalent notions of expansivity have been introduced for flows with singularities, and their terminology is not consistent across the literature, we recall in some detail the definitions and relations that will be used below.
\begin{itemize}
	\item BW-expansive \cite{BW72}: a flow is BW-expansive if for all $\vep>0$ there exists $\delta>0$, such that for all $x,y\in M$ and all continuous functions $C:\RR\to\RR$ with $C(0) = 0$,
	$$
	 d(f_t(x),f_{C(t)}(y))<\delta,\forall t\in\RR \implies y\in f_{[-\vep,\vep]}(x).
	$$
	\item Komuro expansive \cite{Ko,Ko1}: a flow is Komuro expansive if for all $\vep>0$ there exists $\delta>0$, such that for all $x,y\in M$ and all orientation-preserving homeomorphisms $\theta:\RR\to\RR$ satisfying $\theta(0) = 0$,
 	$$
	d(f_t(x),f_{\theta(t)}(y))<\delta,\forall t\in\RR \implies f_{\theta(t_0)}(y)\in f_{[t_0-\vep,t_0+\vep]}(x) \mbox{ for some $t_0\in\RR$}.
	$$
	This property is sometimes called $K^*$-expansivity in part of the
	literature. We use the term Komuro expansivity in order to avoid the
	inconsistent $K/K^*$ terminology across different sources.
	\item Expansivity: a flow is (kinematic) expansive if for all $\vep>0$ there exists $\delta_K>0$, such that for all $x,y\in M$,
	$$ d(f_t(x),f_t(y))<\delta_K,\forall t\implies y\in f_{[-\vep,\vep]}(x).$$
	\item Rescaling expansive \cite{WW}: a flow is rescaling-expansive if for all $\vep>0$ there exists $\delta_R>0$, such that for all $x,y\in M$ and all continuous, strictly increasing function $I:\RR\to\RR$,
	$$
	d(f_t(x),f_{I(t)}(y))\le \delta_R|X(f_t(x))|,\forall t\in\RR \implies f_{I(t)}(y)\in f_{[t-\vep,t+\vep]}(x) \mbox{ $\forall t\in\RR$}.\footnote{In this case one has $f_{I(0)}(y)\in f_{[-\vep,+\vep]}(x) \iff  f_{I(t)}(y)\in f_{[t-\vep,t+\vep]}(x)\ \forall t\in\RR$ after possibly changing $\vep$. See \cite{RWY}.}
	$$
\end{itemize}
It is clear from the definition that 
$$
\mbox{BW-expansive }\implies \mbox{ Komuro expansive }\implies\mbox{ expansive}.
$$
Here the last implication can be seen by taking $\theta$ to be the identity map. Also, 
$$
 \mbox{ Komuro expansive } \implies \mbox{ rescaling expansive }
$$
due to \cite[Theorem 1.3]{RWY}.\footnote{It is noted in \cite[Page 3180]{WW} that even though their definition does not require $I:\RR\to\RR$ to be surjective, the shadowing condition $d(f_t(x),f_{I(t)}(y))\le\delta_R|X(f_t(x))|$ for small $\delta_R>0$  forces $I$ to be surjective whenever $x$ is a regular point. We will see a similar phenomenon in the Komuro expansivity. See Section \ref{s.5}.}

Note that the shadowing condition in rescaling expansivity is much more restrictive than in Komuro expansivity: the condition $d(f_t(x),f_{I(t)}(y))\le \delta_R|X(f_t(x))|$ ensures that the distance of the two orbits becomes small relative to the (already small) flow speed $|X|$ when they approach a singularity. 

It is worth pointing out that if $X$ is a $C^1$ vector field with singularities approximated by regular orbits, then it cannot be BW-expansive. To see this, take $x$ to be a singularity and $C\equiv 0$, then any nearby regular point $y$ automatically satisfies the condition $d(f_t(x),f_{C(t)}(y))<\delta,\forall t\in\RR$. In Komuro expansivity, the time reparametrization $\theta$ is assumed to be an orientation-preserving homeomorphism from $\RR\to\RR$, that is, a strictly increasing, continuous, surjective function. Here the surjectivity cannot be removed (see the definition of WPOTP in \cite{Ko1}); this can be seen by taking $x$ to be a singularity and $\theta(t) = a\cdot\arctan t$ for $a>0$ sufficiently small, and $y$ a regular point sufficiently close to $x$. Indeed it is proven in \cite{Oka} that without the surjectivity assumption, Komuro expansivity is equivalent to BW-expansivity.
% This is the reason that the function $\theta:\RR\to\RR$ in Komuro expansivity is required to be surjective: it forces $\lim_{t\to\pm\infty}\theta(t) = \pm\infty$ which requires $f_{\theta(t)}(y)$ to genuinely move along the entire orbit of $y$ instead of staying on a local orbit segment. 

The discussion above shows that for $C^1$ flows with a singularity accumulated by regular points, Komuro expansivity is the strongest form of expansivity that one can hope for. Such a result has been established for the geometric Lorenz attractor \cite{Ko,Gu,GW} and for all singular hyperbolic attractors in 3-dimensional manifolds \cite{APPV}. In higher dimensions, expansivity (without any time change, minus an exceptional set with zero measure for any invariant probability measure on $\Lambda$) was proven first for all sectional-hyperbolic flows \cite{PYY21} (see also \cite{AC23} for Komuro expansivity with a restrictive condition) and then generalized to multi-singular hyperbolic flows \cite{PYY25a}.

}

\subsection{Statement of main results}
Our first main result shows that every multi-singular hyperbolic set is
robustly Komuro expansive. No transitivity assumption is required, and
the expansivity constant can be chosen uniformly under small $C^1$
perturbations. After that, we will use it to separate periodic orbits with different periods that closely trace one another, establishing two-sided counting bounds and equidistribution results for periodic orbits. 

\begin{main}\label{m.e}
	Suppose $\Lambda$ is a multi-singular hyperbolic set of the flow $(f_t)_{t\in\mathbb{R}}$. Then the flow restricted to $\Lambda$ is robustly Komuro expansive in the following sense: 	
	there exists a $C^1$ neighborhood $\cU$ of $X$ and an open neighborhood  $U_\Lambda$ of $\Lambda$, such that for every $Y\in\cU$, the flow $f_t^Y$ restricted to the maximal invariant set $\Lambda_Y$ in $U_\Lambda$ is Komuro expansive. Furthermore, given $\vep>0$, the constant $\delta>0$ can be chosen uniformly for all $Y\in \cU.$

\end{main}
It is worth noting that in this theorem we do not assume $\Lambda$ to be (chain)
transitive. A slightly stronger result can be proven which allows non-surjective time reparametrizations $\theta$ once an exceptional set is removed. For the precise statement, see Theorem \ref{m.strongerK}.

Combining Theorem \ref{m.e} with \cite[Theorem B]{WW} and \cite[Theorem 1.3]{RWY}, we immediately obtain the following corollary. 
\begin{corollary}There is a $C^1$ residual set $\cR \subset \xX^1(M)$ such that for any $X\in \cR$ and any non-trivial isolated chain transitive set $\Lambda$, the following conditions are equivalent:
	\begin{enumerate}
		\item $\Lambda$  is Komuro expansive for $X$
		\item $\Lambda$ is rescaling expansive for $X$.
		\item $\Lambda$  is locally star for $X$.
		\item $\Lambda$  is multi-singular hyperbolic for $X$.
	\end{enumerate} 
\end{corollary}

Our second goal is to extend part of the classical periodic-orbit theory
of uniformly hyperbolic flows to the multi-singular hyperbolicity setting. Although we
do not obtain the full prime-orbit asymptotic as in \cite{Margulis}, we prove the sharp
two-sided order $e^{ht}/t$, both in a fixed terminal window and
cumulatively, together with equidistribution to the unique measure of
maximal entropy.

For this purpose, denote by $\cP(\Lambda)$ the set of periodic orbits contained in $\Lambda$ (other than the singularities). For $\gamma\in \cP(\Lambda)$, write $\ell(\gamma)$ its prime period and let 
\[
\mu_\gamma=\frac1{\ell(\gamma)}
\int_0^{\ell(\gamma)}\delta_{f_s(x_\gamma)}\,ds
\]
be the normalized orbit measure, where $x_\gamma\in\gamma$ is arbitrary.
Define
\begin{align*}
	\Pi_\Lambda(t)&=\{\gamma\in\cP(\Lambda):\ell(\gamma)\le t\},\\
	\Pi_\Lambda(t,\Delta)&=\{\gamma\in\cP(\Lambda):
	t-\Delta<\ell(\gamma)\le t\}.
\end{align*}
The next main result is the following fixed window version of asymptotic counting and equidistribution.
\begin{main}\label{m.counting}
	Let $\Lambda$ be a multi-singular chain recurrence class of a $C^1$ vector field $X$ on a compact Riemannian manifold $M$. %Further assume that all singularities in $\Lambda$ are non-degenerate and active, 
	Further assume all periodic orbits in $\Lambda$ are homoclinically related, and $h:= h_{\mathrm{top}}(X|_\Lambda)>0$. Then, there exist constants $\Delta>0$, $C>1$, and $t_0>0$ such that, for every
	$t\ge t_0$,
	\begin{equation}\label{eq.window-counting}
		C^{-1}\frac{e^{ht}}{t}
		\le \#\Pi_\Lambda(t,\Delta)
		\le C\frac{e^{ht}}{t}.
	\end{equation}
	Moreover, by increasing $C$ if necessary, it holds
	\begin{equation}\label{eq.cumulative-counting}
		C^{-1}\frac{e^{ht}}{t}
		\le \#\Pi_\Lambda(t)
		\le C\frac{e^{ht}}{t}.
	\end{equation}
	Finally, the periodic orbit measures in the fixed window equidistribute to
	$\mu_{MME}$, the unique measure of maximal entropy:
	\begin{equation}\label{eq.window-equidistribution}
		\nu_{t,\Delta}:=
		\frac1{\#\Pi_\Lambda(t,\Delta)}
		\sum_{\gamma\in\Pi_\Lambda(t,\Delta)}
		\mu_\gamma
		\xrightarrow[t\to\infty]{w^*}\mu_{MME}
	\end{equation}
	where the convergence is in the weak-* topology. 
	The same conclusion holds for the cumulative averages
	\[
	\bar\nu_t:=\frac1{\#\Pi_\Lambda(t)}
	\sum_{\gamma\in\Pi_\Lambda(t)}
	\mu_\gamma.
	\]
	
\end{main}

Indeed the conclusion of Theorem \ref{m.counting} holds robustly in a $C^1$ neighborhood. More precisely:
\begin{main}\label{m.robust.counting} There exists a $C^1$ residual set $\cR\subset \xX^1(M)$, such that for every $X\in \cR$, every non-trivial, isolated chain recurrence class $\Lambda$ that is multi-singular hyperbolic, there exists a $C^1$ neighborhood $\cU$ of $X$ and an open neighborhood $U$ of $\Lambda$, such that for every $Y\in\cU$ and its maximal invariant set $\Lambda_Y$ in $U$, the conclusions of Theorem \ref{m.counting} hold. \end{main}

% Here we do not claim the constants $\Delta,C$ to be uniform in $\cU$. 

\subsection{Application to star flows}

In this section we apply Theorem \ref{m.e} and \ref{m.counting} to star vector fields. Recall that a vector field is said to have the star property, if there exists a $C^1$ neighborhood $\cU$ of $X$, such that for every $Y\in\cU$, every critical element (i.e., periodic orbit and singularity) of $Y$ is hyperbolic. Denote by $\xX^1_*(M)$ the set of all star vector fields on $M$. 
\begin{main}\label{m.e.star}
	There exists a $C^1$ open and dense subset $\cO\subset \xX^1_*(M)$ such that every $X\in \cO$ is robustly Komuro expansive on its chain recurrence set $\CR(X)$.
\end{main}

Next, we consider the counting and equidistribution of periodic orbits. Let $\Pi(t)$ (resp.\ $\Pi(t,\Delta)$) denote the set of all periodic orbits of $X$ with prime period at most $t$ (resp. in $(t-\Delta,t]$)
and $\mu_\gamma$ be defined as before. This time we consider all periodic orbits on $M$, not just those in $\Lambda$. 

\begin{main}\label{m.counting.star}
	There exists a $C^1$ open and dense subset $\cO$ in  $\xX^1_*(M)$ such that for every $X\in \cO$ with topological entropy $h = h_{\mathrm{top}}(X) >0$, the following hold.
	
	There exist constants $\Delta>0$, $C>1$, and $t_0>0$ such that, for every
	$t\ge t_0$,
	\begin{equation*}
		C^{-1}\frac{e^{ht}}{t}
		\le \#\Pi(t,\Delta)
		\le C\frac{e^{ht}}{t}.
	\end{equation*}
	Moreover, by increasing $C$ if necessary, it holds
	\begin{equation*}
		C^{-1}\frac{e^{ht}}{t}
		\le \#\Pi(t)
		\le C\frac{e^{ht}}{t}.
	\end{equation*}
	Finally, the periodic orbit measures in the fixed window equidistribute to
	$\mu_{MME}$, the unique measure of maximal entropy of $X$:
	\begin{equation*}
		\nu_{t,\Delta}:=
		\frac1{\#\Pi(t,\Delta)}
		\sum_{\gamma\in\Pi(t,\Delta)}
		\mu_\gamma
		\xrightarrow[t\to\infty]{w^*}\mu_{MME}
	\end{equation*}
	where the convergence is in the weak-* topology. 
	The same conclusion holds for the cumulative averages
	\[
	\bar\nu_t:=\frac1{\#\Pi(t)}
	\sum_{\gamma\in\Pi(t)}
	\mu_\gamma.
	\]
\end{main}
Here the uniqueness of the MME is due to \cite[Theorem A]{PYY25a}.

\section{Preliminaries}

Denote by $\Reg(X)$ the set of regular points of $X$, that is, $\Reg(X) = M\setminus \Sing(X)$. 
Given $x\in\Reg(X)$, for each $\rho>0$, let 
\[N_{\rho}(x)=\{v\in N(x): |v|\le\rho\}.\]
If $\rho$ does not exceed the injectivity radius of $M$, we define
\[\cN_{\rho}(x)=\exp_x(N_{\rho}(x)),\]
where $\exp_x:T_xM\to M$ is the exponential map at $x$. The submanifold $\cN_{\rho}(x)$ will be called a {\em normal manifold} at $x$. Along every regular orbit from $x$ to $f_t(x)$, the holonomy of flow induces a {\em sectional Poincar\'e map} from $\cN_{\rho}(x)$ to $\cN_{\rho}(f_t(x))$, which will be denoted as $\cP_{t,x}$. By the implicit function theorem, the domain of the map $\cP_{t,x}$ contains a small neighborhood of $x$ in $\cN_{\rho}(x)$. Precisely, we have the following result.

\begin{lemma}[{\cite[Lemma 2.2]{GY}}]
\label{lem.liao-size}
Given $X\in \xX^1(M)$, there exist constants $\overline\rho_0>0$ and $K_0>1$ such that for every $\rho\in(0,\overline\rho_0]$ and for every regular point $x\in M$, the holonomy map 
\[\cP_{1,x}:\cN_{\rho K_0^{-1}|X(x)|}(x)\to\cN_{\rho|X(f_1(x))|}(f_1(x))\]
is well-defined, differentiable, and injective.  
\end{lemma}

For each $x\in\mathrm{Reg}(X)$ and $t\in [0,1]$, the map $\cP_{t,x}$ can be lifted to a map $P_{t,x}$ on the normal bundle via the exponential map as follows:
\[P_{t,x}=\exp_{x_t}^{-1}\circ \cP_{t,x}\circ \exp_x.\]
Note that the domain of $P_{t,x}$ contains $N_{\overline\rho_0K_0^{-1}|X(x)|}(x)$. One then defines the {\em scaled sectional Poincar\'e map on the normal bundle}:
\[P_{t,x}^*: N_{\overline\rho_0K_0^{-1}}(x)\to N_{\overline\rho_0}(x_t),\quad P_{t,x}^*(y):=\frac{P_{t,x}(y|X(x)|)}{|X(x_t)|}.\]
As $D_x\cP_{t,x}=D_0P_{t,x}=\psi_{t,x}$, one has $D_0P_{t,x}^*=\psi^*_{t,x}$.

Reducing $\overline\rho_0>0$ if necessary, one can define the {\em local Poincar\'e map} at each regular point $x$, denoted by $\cP_x$, which maps every point $y\in B_{\overline\rho_0|X(x)|}(x)$ to the unique point of intersection between $\cN_{\overline\rho_0|X(x)|}(x)$ and the local flow line at $y$. Here, for each $\rho>0$, $B_{\rho}(x)$ denotes the $\rho$-neighborhood of $x$. 
By \cite[Proposition 2.2]{WW}, the local Poincar\'e map $\cP_x:B_{\overline\rho_0|X(x)|}(x)\to \cN_{\overline\rho_0|X(x)|}(x)$  can be defined through a flow-box map 
\[F_x: U_{\overline\rho_0|X(x)|}(x)\to M,\quad F_x(v+tX(x))=f_t(\exp_x(v)),\] 
where 
\[U_{\overline\rho_0|X(x)|}(x)=\{v+tX(x)\in T_xM: v\in N_{\overline\rho_0|X(x)|}(x), |t|<\overline\rho_0\}.\]
Then, for each $y=F_x(v+tX(x))\in B_{\rho_0|X(x)|}(x)$, one defines
\[\cP_x(y)=F_x(v)=\exp_x(v).\]
Moreover, assuming $\overline\rho_0$ is small enough, one has
\[m(D_pF_x)\ge \frac{1}{3}\quad \text{and}\quad \|D_pF_x\|\le 3,\]
for every $p\in U_{\overline\rho_0|X(x)|}(x)$, where $m(A)$ denotes the mininorm of the linear operator $A$. 
One can find more related discussions in \cite[Section 2.3]{PYY25}.

\section{Fake foliations and coordinate systems}
In this section we review the fake foliation coordinate systems first constructed in \cite{PYY25} and \cite{PYY25a}.
Let $\Lambda$ be a multi-singular hyperbolic set, with a dominated splitting $N_{\Lambda}=N^{cs}\oplus N^{cu}$ for the linear Poincar\'e flow. Note that an $(\eta, T)$-dominated splitting $N_{\Lambda}=N^{cs}\oplus N^{cu}$ with respect to the linear Poincar\'e flow $(\psi_t)_{t\in\mathbb{R}}$ is also an $(\eta, T)$-dominated splitting with respect to the scaled linear Poincar\'e flow $(\psi^*_t)_{t\in\mathbb{R}}$.

\subsection{A coordinate system along regular orbits}\label{sect.coord-reg}
For each $x\in \Lambda\setminus\Sing(X)$ and $v\in N(x)$, there is a unique decomposition $v=v^{cs}+v^{cu}$ with $v^{cs}\in N^{cs}(x)$ and $v^{cu}\in N^{cu}(x)$. Given $\beta>0$, one defines on  $N(x)$ the following cones:
\begin{gather*}
	C^{cs}_{\beta}(x)=\{v=v^{cs}+v^{cu}\in N(x): |v^{cu}|<\beta|v^{cs}|\},\\
	C^{cu}_{\beta}(x)=\{v=v^{cs}+v^{cu}\in N(x): |v^{cs}|<\beta|v^{cu}|\}.
\end{gather*}
We translate these cones to each point on $N(x)$, obtaining cones on $N(x)$. Abusing notations, we will denote these cones also by $C^{cs}_{\beta}$ and $C^{cu}_{\beta}$, respectively.

Recall that the scaled sectional Poincar\'e map $P^*_{1,x}$ is defined on $N_{\overline\rho_0K_0^{-1}}(x)$ with a uniform size and $D_0P_{1,x}^*=\psi^*_{1,x}$. Thus, we can extend its definition to the entire normal space $N(x)$, and make sure that it remains $C^1$ close to $\psi^*_{1,x}$. 
Precisely, given $\iota>0$, one can take $\nu>0$ small and $K>0$ large, and define on $N(x)$ a $C^1$ map $\tilde{P}_{1,x}^*$ such that
\begin{equation}
	\tilde{P}_{1,x}^*=\left\{\begin{array}{ll}
		P^*_{1,x}, & |y|<\nu,\\
		\psi^*_{1,x}, & |y|>K\nu,
	\end{array}\right.
\end{equation}
and  
\[\|\tilde{P}_{1,x}^*-\psi^*_{1,x}\|_{C^1}<\iota.\]
Moreover, the constants $\nu$ and $K$ can be taken independent of $x$.

Now, for any $\beta>0$, one can require $\iota>0$ to be small enough and apply the Hadamard-Perron Theorem \cite[Theorem  6.2.8]{KH} to the family of maps $\tilde{P}_{1,x}^*$, thus obtaining two foliations $\cF^{cs,*}_{x,N}$, $\cF^{cu,*}_{x,N}$ on each normal space $N(x)$, such that the following properties hold:
\begin{itemize}
	\item (Tangent to $\beta$-cones) Each leaf of $\cF^{cs,*}_{x,N}$ is the graph of a $C^1$ function $h_x^E:N^{cs}(x)\to N^{cu}(x)$ with $\|Dh^E_x\|_{C^0}<\beta$; similarly, each leaf of $\cF^{cu,*}_{x,N}$ is the graph of a $C^1$ function $h_x^F:N^{cu}(x)\to N^{cs}(x)$ with $\|Dh^F_x\|_{C^0}<\beta$. In other words, the leaves of $\cF^{\xi,*}_{x,N}$ ($\xi=cs, cu$) are tangent to the $\beta$-cone $C^{\xi}_{\beta}$.
	\item (Invariance) For each $y\in N(x)$, it holds that 
	\[\tilde{P}^*_{1,x}(\cF^{\xi,*}_{x,N}(y))=\cF^{\xi,*}_{x_1,N}(\tilde{P}^*_{1,x}(y)),\quad\xi=cs, cu.\] 
	\item (Local product structure) For each $y\in N(x)$, there exist a unique $y^E\in\cF^{cs,*}_{x,N}(x)$ and a unique $y^F\in\cF^{cu,*}_{x,N}(x)$ such that 
	\[\{y\}=\cF^{cu,*}_{x,N}(y^E)\pitchfork \cF^{cs,*}_{x,N}(y^F).\]
\end{itemize}
We see from the local product structure that for each $x\in \Lambda\setminus\Sing(X)$, the foliations $\cF^{\xi,*}_{x,N}$ ($\xi=cs, cu$) form a coordinate system of $N(x)$. For any $y,y'\in N(x)$, let us denote $[y,y']_N$ to be the unique intersection of $\cF^{cu,*}_{x,N}(y)$ and $\cF^{cs,*}_{x,N}(y')$, i.e. 
\[\{[y,y']_N\}=\cF^{cu,*}_{x,N}(y)\pitchfork \cF^{cs,*}_{x,N}(y').\]
Then there is a constant $\rho_0\in (0,\overline\rho_0)$ such that for any $y,y'\in N_{\rho_0}(x)$, it holds
\[[y,y']_N\in N_{\overline\rho_0K_0^{-1}}(x).\]

By the construction of the map $\tilde{P}^*_{1,x}$, the dynamics near the base point is given by the scaled sectional Poincar\'e map $P^*_{1,x}$, which has a domain $N_{\rho_0K_0^{-1}}(x)$ of uniform size. Therefore, for the given $\beta>0$ and $\iota>0$, we may assume without loss of generality that $\rho_0>0$ is small enough such that 
\begin{equation}\label{eq.assumption-on-rho0}
	\rho_0K_0^{-1}<\nu.
\end{equation}
Then, we can rescale the foliations $\cF^{\xi,*}_{x,N}$ ($\xi=cs, cu$) by the flow speed $|X(x)|$ and push them to the manifold $M$ via the exponential map $\exp_x$. We shall obtain in this way foliations that reflect the dynamics of the sectional Poincar\'e map on the manifold.

Precisely, let 
\[S_x:N(x)\to N(x),\quad S_x(u)=|X(x)|u,\]
and define
\[\cF^{\xi}_{x,\cN}=\exp_x\left(S_x\left(\cF^{\xi,*}_{x,N}\right)\right),\quad \xi=cs, cu.\]
Of course, these foliations only make sense in the neighborhood $\cN_{\varrho}(x)$, where $\varrho>0$ is the injectivity radius of the exponential map. 
For our purpose, we consider these foliations only on $\cN_{\rho_0|X(x)|}(x)$, i,e. for each $y\in \cN_{\rho_0|X(x)|}(x)$, let $\tilde{y}=S_x^{-1}\circ\exp_{x}^{-1}(y)$ and 
\begin{equation}\label{eq.ff-construction}
	\cF^{\xi}_{x,\cN}(y)=\exp_x\left(S_x\left(\cF^{\xi,*}_{x,N}(\tilde{y})\cap N_{\rho_0}(x)\right)\right),\quad \xi=cs, cu.
\end{equation}
%Observe that the preimage of each $\cF^{\xi}_{x,\cN}(y)$ ($\xi=cs, cu$) under the exponential map $\exp_x$ remains tangent to the $\alpha$-cone $C^{\xi}_{\alpha}$. 

\begin{proposition}\label{prop.ff-regular}
The following properties hold for the foliations $\cF^{cs}_{x,\cN}$ and $\cF^{cu}_{x,\cN}$: 
\begin{itemize}
	\item (Tangent to $\beta$-cones) The preimage of each $\cF^{\xi}_{x,\cN}(y)$ ($\xi=cs, cu$) under the exponential map $\exp_x$ is tangent to the $\beta$-cone $C^{\xi}_{\beta}$. 
	\item (Local invariance) For any $y\in \cN_{\rho_0K_0^{-1}|X(x)|}(x)$ and $\xi=cs, cu$, one has
	\[
		\cP_{1,x}\left(\cF^{\xi}_{x,\cN}(y)\cap\cN_{\rho_0K_0^{-1}|X(x)|}(x)\right)\subset\cF^{\xi}_{x_1,\cN}(\cP_{1,x}(y)),\quad \text{and}\]
		\[\cP_{1,x}^{-1}\left(\cF^{\xi}_{x_1,\cN}(y)\cap\cN_{\rho_0K_0^{-1}|X(x_1)|}(x_1)\right)\subset\cF^{\xi}_{x,\cN}(\cP_{1,x}^{-1}(y)).\]
		
	\item (Local product structure) For any $y\in \cN_{\rho_0|X(x)|}(x)$, there is a unique $y^E\in\cF^{cs}_{x,\cN}(x)$ and a unique $y^F\in \cF^{cu}_{x,\cN}(x)$ such that 
	\[\{y\}=\cF^{cu}_{x,\cN}(y^E)\pitchfork\cF^{cs}_{x,\cN}(y^F).\]
\end{itemize}
\end{proposition}
\begin{proof}
	The first item follows from the fact that the foliations $\cF^{\xi,*}_{x,N}(y)$ are tangent respectively to $\beta$-cones $C^{\xi}_{\beta}$ and that the scaling map $S_x$ preserves cones. The local invariance property follows from the corresponding property of the foliations $\cF^{\xi,*}_{x,N}$ ($\xi=cs, cu$), and the fact that 
	\[\cP_{1,x}\circ(\exp_{x}\circ S_x)=(\exp_{x_1}\circ S_{x_1})\circ\tilde{P}^*_{1,x},\]
	restricted to the domain $N_{\rho_0K_0^{-1}}(x)$. 
	The local product structure is a consequence of the same property of the foliations $\cF^{\xi,*}_{x,N}$. Note that our choice of $\rho_0>0$ ensures that $y^E$ and $y^F$ are contained in $\cN_{\overline\rho_0K_0^{-1}|X(x)|}(x)$.
\end{proof}

For any $x\in\Lambda\setminus\Sing(X)$ and $y\in\cN_{\rho_0|X(x)|}(x)$, we have that 
\[\{y^F\}=\cF^{cs}_{x,\cN}(y)\pitchfork\cF^{cu}_{x,\cN}(x).\]
Let $d_x^E(y)$ be the distance between $y^F$ and $y$ along the leaf $\cF^{cs}_{x,\cN}(y)$, i.e. 
\begin{equation}\label{eq.E-length-def}
	d^E_x(y):=d_{\cF^{cs}_{x,\cN}}(y^F,y).
\end{equation}
Similarly, let $d_x^F(y)$ be the distance between $x$ and $y^F$ along the leaf $\cF^{cu}_{x,\cN}(x)$, i.e.
\begin{equation}\label{eq.F-length-def}
	d^F_x(y):=d_{\cF^{cu}_{x,\cN}}(x,y^F).
\end{equation}
We refer to $d^E_x(y)$ and $d^F_x(y)$ as the {\em $E$-length} and {\em $F$-length} of $y$, respectively. 

We can consider also the following distances:
\[d_{\cF^{cs}_{x,\cN}}(x,y^E),\quad d_{\cF^{cu}_{x,\cN}}(y^E,y),\]
where $y^E$ is the unique intersection of $\cF^{cs}_{x,\cN}(x)$ and $\cF^{cu}_{x,\cN}(y)$. 
\begin{corollary}\label{cor.switch-base}
	For any $\beta>0$ small enough, the following properties hold:
	\begin{itemize}
		\item if $d_x^E(y)\ge d^F_x(y)$, then $d_{\cF^{cs}_{x,\cN}}(x,y^E)\ge 0.99 d_{\cF^{cu}_{x,\cN}}(y^E,y)$; 
		\item if $d_x^F(y)\ge d^E_x(y)$, then $d_{\cF^{cu}_{x,\cN}}(y^E,y)\ge 0.99 d_{\cF^{cs}_{x,\cN}}(x,y^E)$.
	\end{itemize}
\end{corollary}
\begin{proof}
	In view of \eqref{eq.ff-construction}, we consider firstly the foliations $\cF^{cs,*}_{x,N}$ and $\cF^{cu,*}_{x,N}$ on the normal space. Let 
	\[d^{E,*}_{x}(\tilde{y})=d_{\cF^{cs,*}_{x,N}}(\tilde{y}^F,\tilde{y}), \quad d^{F,*}_{x}(\tilde{y})=d_{\cF^{cu,*}_{x,N}}(x,\tilde{y}^F).\]
	By construction, the leaves of $\cF^{\xi,*}_{x,N}$ are tangent to the $\beta$-cone $C^{\xi}_{\beta}$, for $\xi=cs, cu$. Reducing $\beta>0$ if necessary, one can see that if $d^{E,*}_x(\tilde{y})\ge d^{F,*}_x(\tilde{y})$, then 
	\[d_{\cF^{cs,*}_{x,N}}(x,\tilde{y}^E)\ge 0.999 d_{\cF^{cu,*}_{x,N}}(\tilde{y}^E,y).\]
	Note that this relation is preserved by the scaling map $S_x$. Also, the differential of the exponential map $\exp_x$, restricted to $B_{\rho_0|X(x)|}(x)$, is arbitrarily close to the identity map, where $\rho_0>0$ is small enough assuming $\beta$ is small. See the assumption \eqref{eq.assumption-on-rho0}. 
	Now, the first item of the corollary follows from the previous inequality and \eqref{eq.ff-construction}. The second item can be proved similarly.
\end{proof}

\subsection{A coordinate system near singularities}

Let $\sigma$ be a singularity in $\Lambda$, which is either Lorenz-like or reversed Lorenz-like. 
We assume firstly that $\sigma$ is Lorenz-like. In particular, there is a dominated splitting $T_{\sigma}M=E^{ss}_{\sigma}\oplus E^c_{\sigma}\oplus E^u_{\sigma}$ for the tangent flow such that $\dim E^c_{\sigma}=1$ and $E^{ss}_{\sigma}\oplus E^c_{\sigma}$ corresponds to the stable subspace. 
Without loss of generality, let us assume that $E^{ss}_{\sigma}$, $E^c_{\sigma}$ and $E^u_{\sigma}$ are pairwise orthogonal. We will denote $E^{cs}_{\sigma}:=E^{ss}_{\sigma}\oplus E^c_{\sigma}$ and $E^{cu}_{\sigma}:=E^c_{\sigma}\oplus E^u_{\sigma}$.
Given $\alpha>0$, one can define $\alpha$-cones around each subbundle $E^{ss}_{\sigma}, E^c_{\sigma}, E^u_{\sigma}$, $E^{cs}_{\sigma}$ and $E^{cu}_{\sigma}$, and extend the cones to $T_{\sigma}M$ by translation. 

On the tangent space $T_{\sigma}M$, we consider a $C^1$ diffeomorphism $g_{\sigma}$ such that 
\begin{equation*}
	g_{\sigma}(u)=\left\{\begin{array}{ll}
		\exp_{\sigma}^{-1}\circ f_1\circ\exp_{\sigma}, & |u|<\nu\\
		Df_1|_{T_{\sigma}M}, & |u|>K\nu
	\end{array}\right.
\end{equation*}
where $\nu>0$ and $K>1$. 
Note that $Df_1|_{T_{\sigma}M}=Df_1|_{T_{\sigma}M}$. For any $\iota>0$, we can take $\nu$ small and $K$ large enough such that $\|g_{\sigma}-Df_1|_{T_{\sigma}M}\|_{C^1}<\iota$. 

Now, given $\alpha>0$, one can require $\iota>0$ to be small enough and apply the Hadamard-Perron Theorem \cite[Theorem 6.2.8]{KH} to the map $g_{\sigma}$, obtaining foliations $\cF^{\xi}_{\sigma}$ ($\xi=ss,c,u,cs,cu$) with the following properties:
\begin{itemize}
	\item Each leaf of $\cF^{\xi}_{\sigma}$  ($\xi=ss,c,u,cs,cu$) is tangent to the corresponding $\alpha$-cones. 
	\item For any $z\in T_{\sigma}M$, one has 
	\[g_{\sigma}(\cF^{\xi}_{\sigma}(z))=\cF^{\xi}_{\sigma}(g_{\sigma}(z)),\quad \xi =ss,c,u,cs,cu.\]
	\item The foliations $\cF^{ss}_{\sigma}$ and $\cF^{cu}_{\sigma}$ form a product structure: for any $y,z\in T_{\sigma}M$, $\cF^{ss}_{\sigma}(y)$ and $\cF^{cu}_{\sigma}(z)$ intersect transversely at a unique point. The same holds for the pair of foliations $\cF^{cs}_{\sigma}$ and $\cF^{u}_{\sigma}$.
	\item $\cF^{ss}_{\sigma}$ and $\cF^c_{\sigma}$  subfoliate of $\cF^{cs}_{\sigma}$, and $\cF^c_{\sigma}$, $\cF^u_{\sigma}$ subfoliate $\cF^{cu}_{\sigma}$. Moreover, inside each leaf of $\cF^{cs}_{\sigma}$, the foliations $\cF^{ss}_{\sigma}$ and $\cF^c_{\sigma}$ form a product structure; similarly, inside each leaf of $\cF^{cu}_{\sigma}$, the foliations $\cF^c_{\sigma}$, $\cF^u_{\sigma}$ form a product structure. 
\end{itemize}

One then projects these foliations in $T_{\sigma}M(\nu)$ to the manifold $M$ via the exponential map $\exp_{\sigma}$, obtaining locally invariant foliations in the neighborhood $B_{\nu}(\sigma)$ of $\sigma$. Abusing notations, we shall denote also by $\cF^{\xi}_{\sigma}$ ($\xi=ss,c,u,cs,cu$) for the foliations on $B_{\nu}(\sigma)$. Note that the $\alpha$-cones on $T_{\sigma}M(\nu)$ can be projected to $B_{\nu}(\sigma)$ via the exponential map $\exp_{\sigma}$. We denote these cones by $C_{\sigma,\alpha}^{\xi}$ ($\xi=ss,c,u,cs,cu$).   
\begin{proposition}\label{prop.ff-singularity}
	Let $\sigma\in\Lambda\cap\Sing(X)$ be a Lorenz-like singularity. Given $\alpha>0$, there exist $r^*>r_0>0$ such that the neighborhood $B_{r^*}(\sigma)$ admits foliations $\cF^{\xi}_{\sigma}$ ($\xi=ss,c,u,cs,cu$) with the following properties:
	\begin{itemize}
		\item For each $x\in B_{r^*}(\sigma)$ and $\xi=ss,c,u,cs,cu$, the leaf $\cF^{\xi}_{\sigma}(x)$  is $C^1$ and tangent to the $\alpha$-cone $C_{\sigma,\alpha}^{\xi}$.
		\item For each $x\in B_{r_0}(\sigma)$ and $\xi=ss,c,u,cs,cu$, one has
		\[f_1(\cF^{\xi}_{\sigma}(x)\cap B_{r_0}(\sigma))\subset \cF^{\xi}_{\sigma}(x_1),\quad f_1^{-1}(\cF^{\xi}_{\sigma}(x)\cap B_{r_0}(\sigma))\subset \cF^{\xi}_{\sigma}(x_{-1}).\]
		\item The foliations $\cF^{ss}_{\sigma}$, $\cF^{cu}_{\sigma}$ form a local product structure on $B_{r_0}(\sigma)$: for any $y,z\in B_{r_0}(\sigma)$, $\cF^{ss}_{\sigma}(y)$ and $\cF^{cu}_{\sigma}(z)$ intersect transversely at a unique point. The same holds for the pair of foliations $\cF^{cs}_{\sigma}$ and $\cF^{u}_{\sigma}$.
		\item $\cF^{ss}_{\sigma}$ and $\cF^c_{\sigma}$  subfoliate of $\cF^{cs}_{\sigma}$, and $\cF^c_{\sigma}$, $\cF^u_{\sigma}$ subfoliate $\cF^{cu}_{\sigma}$. Moreover, inside each leaf of $\cF^{cs}_{\sigma}$, the foliations $\cF^{ss}_{\sigma}$ and $\cF^c_{\sigma}$ form a local product structure on $B_{r_0}(\sigma)$; similarly, inside each leaf of $\cF^{cu}_{\sigma}$, the foliations $\cF^c_{\sigma}$, $\cF^u_{\sigma}$ form a local product structure on $B_{r_0}(\sigma)$. 
	\end{itemize}
\end{proposition}
\begin{proof}
	We take $r^*=\nu$ and $r_0\in (0,r^*)$ small enough. Note that restricted to $T_{\sigma}M(\nu)$, one has 
	\[\exp_{\sigma}\circ g_{\sigma}=f_1\circ \exp_{\sigma}.\]
	Then these properties of the foliations $\cF^{\xi}_{\sigma}$ ($\xi=ss,c,u,cs,cu$) follow immediately from that of the foliations on $T_{\sigma}M(\nu)$.
\end{proof}

Suppose now $\sigma$ is reversed Lorenz-like. Then there is a dominated splitting of the tangent bundle $T_{\sigma}M=E^s\oplus E^c\oplus E^{uu}$, where $\dim E^c=1$ and $E^c\oplus E^{uu}$ corresponds to the unstable subspace of $\sigma$. 
A similar construction gives locally invariant foliations $\cF^{\xi}_{\sigma}$ ($\xi=s,c,uu,cs,cu$)
in a small neighborhood $B_{r^*}(\sigma)$, with similar properties as in the preceding proposition.

\section{Estimates near singularities}\label{sect.est-sing}

In this section, we focus on the local dynamics near singularities of a multi-singular hyperbolic set $\Lambda$. The main purpose is to understand how expansion can be ensured for regular orbits segments passing by a singularity.

Let $\alpha>0$ be a constant, which will be chosen to be small enough. 
Since there are only finitely many singularities in $\Lambda$, we may assume that the constants $r^*>r_0>0$ given by Proposition \ref{prop.ff-singularity} are valid for all singularities in $\Lambda$. 
For each $r\in (0,r_0)$ and $\sigma\in\Sing(X)\cap\Lambda$, we define $W_r(\sigma)$ to be the set of points $x\in B_{r_0}(\sigma)$ which has an orbit segment $f_{[-t^-,t^+]}(x)$ containing $x$ such that it is contained entirely in $B_{r_0}(\sigma)$ and it intersects $B_r(\sigma)$. Beware that $W_r(\sigma)$ depends on the fixed constant $r_0>0$. 
For any $x\in W_r(\sigma)$, if $x\notin W_{loc}^u(\sigma)$, let $t_x^->0$ be the largest number such that $f_{(-t_x^-,0]}(x)\subset B_{r_0}(\sigma)$; and if $x\notin W_{loc}^s(\sigma)$, let $t_x^+>0$ be the largest number such that $f_{[0,t_x^+)}(x)\subset B_{r_0}(\sigma)$. Define 
\[\partial W^-_r(\sigma)=\{f_{-t_x^-}(x):\ x\in W_r(\sigma)\setminus W^u_{loc}(\sigma)\},\] 
and 
\[\partial W^+_r(\sigma)=\{f_{t_x^+}(x):\ x\in W_r(\sigma)\setminus W^s_{loc}(\sigma)\}.\] 
For a point $x\in \partial W^-_r(\sigma)\setminus W^s_{loc}(\sigma)$, let $t_x>0$ be the largest number such that $f_{(0,t_x)}(x)\subset B_{r_0}(\sigma)$. We write $x^+=f_{t_x}(x)$, which is contained in $\partial W^+_r(\sigma)$. 
For $\delta>0$ small, suppose $y\in \cN_{\rho_0|X(x)|}(x)$ such that, for some time reparametrization (i.e., an orientation-preserving homeomorphism of $\RR$) $\theta:\mathbb{R}\to\mathbb{R}$ with $\theta(0)=0$, it holds 
\begin{equation}\label{eq.delta-close-condition}
	d(f_t(x),f_{\theta(t)}(y))<\delta,\quad \forall t\in [0, t_x].
\end{equation}
Let $y^+=\cP_{x^+}(f_{\theta(t_x)}(y))$, which is contained in $\cN_{\rho_0|X(x^+)|}(x^+)$ (by reducing $\delta$ if necessary). 
Without loss of generality, we assume that $y^+=f_{t^+}(y)$ with $0<t_x\le t^+$. As a consequence of Proposition \ref{prop.ff-regular}, we have the $E$-length $d_x^E(y)$ and $F$-length $d_x^F(y)$ of $y$ on $\cN_{\rho_0|X(x)|}(x)$, as defined respectively by \eqref{eq.E-length-def} and \eqref{eq.F-length-def}. Similarly, we have the $E$-length $d_{x^+}^E(y^+)$ and the $F$-length $d_{x^+}^F(y^+)$ for $y^+$.
We will consider the following two cases:
\begin{enumerate}
	\item $d_{x^+}^E(y^+)\ge d^F_{x^+}(y^+)$, i.e. upon leaving $W_r(\sigma)$ the $E$-length of $y^+$ is larger than or equal to its $F$-length;
	\item $d_{x}^F(y)\ge d^E_{x}(y)$, i.e. upon entering $W_r(\sigma)$ the $F$-length of $y$ is larger than or equal to its $E$-length.
\end{enumerate}

In the first case, we will show backward expansion on the $E$-length and it must be larger than or equal to $F$-length when the orbit of $y$ enters $W_r(\sigma)$, as given by the following lemma.
\begin{lemma}\label{lem.E-large}
For all $r_0>0$ sufficiently small, there exist constants $\bar{r}\in(0,r_0)$,  such that for $0<r<\bar{r}$, there exists 	$\delta_0>0$ such that for all $0<\delta<\delta_0$, the following property holds for every $\sigma\in\Sing(X)\cap\Lambda$:

Let $x\in(\partial W^-_r(\sigma)\setminus W^s_{loc}(\sigma))\cap\Lambda$ and $x^+=f_{t_x}(x)$. Suppose $y\in\cN_{\rho_0|X(x)|}(x)$ such that \eqref{eq.delta-close-condition} holds for some continuous, strictly increasing function $\theta:\RR\to\RR$ (not necessarily surjective) with $\theta(0) = 0$.  Suppose that $y^+=\cP_{x^+}(f_{\theta(t_x)}(y))$ and $y^+=f_{t^+}(y)$ with $0<t_x\le t^+$, and moreover, 
\begin{equation}\label{eq.E-large-assumption}
	d^E_{x^+}(y^+)\ge d^F_{x^+}(y^+), \quad (\text{$y^+$ has large $E$-length at $x^+$})
\end{equation}
then we have the following conclusions:
\begin{enumerate}
	\item ($y$ has large $E$-length at $x$) $d^E_x(y)\ge d^F_x(y)$.
	\item (backward exponential expansion on $E$-length) there exists $\lambda_{\sigma}>1$ such that
	\[d^E_x(y)\ge \lambda_{\sigma}^{t_x}d^E_{x^+}(y^+).\]
\end{enumerate}
\end{lemma}

In the second case, we show forward expansion of the $F$-length and that the $F$-length must be larger than or equal to $E$-length when the orbit of $y$ leaves $W_r(\sigma)$. This is given by Lemma \ref{lem.F-large}.

\begin{lemma}\label{lem.F-large}
For all $r_0>0$ sufficiently small, there exist constants $\bar{r}\in(0,r_0)$, such that for $0<r<\bar{r}$, there exists	$\delta_0>0$ such that for all $0<\delta<\delta_0$, the following property holds for every $\sigma\in\Sing(X)\cap\Lambda$:

Let $x\in(\partial W^-_r(\sigma)\setminus W^s_{loc}(\sigma))\cap\Lambda$ and $x^+=f_{t_x}(x)$. Suppose $y\in\cN_{\rho_0|X(x)|}(x)$ such that \eqref{eq.delta-close-condition} holds for some continuous, strictly increasing function $\theta:\RR\to\RR$ (not necessarily surjective) with $\theta(0) = 0$.  Suppose $y^+=\cP_{x^+}(f_{\theta(t_x)}(y))$ and $y^+=f_{t^+}(y)$ with $0<t_x\le t^+$, and moreover, 
\begin{equation}\label{eq.F-large-assumption}
	d^F_x(y)\ge d^E_x(y),\quad (\text{$y$ has large $F$-length at $x$})
\end{equation}
then we have the following conclusions:
\begin{enumerate}
	\item ($y^+$ has large $F$-length at $x^+$) $d^F_{x^+}(y^+)\ge d^E_{x^+}(y^+)$;
	\item (forward expansion on $F$-length) there exists $\lambda_{\sigma}>1$ such that 
	\[d^F_{x^+}(y^+)\ge \lambda_{\sigma}^{t_x}d^F_x(y).\]
\end{enumerate}
\end{lemma}

\begin{remark}\label{r.EF}
	The assumption \eqref{eq.E-large-assumption} can be replaced by $d^E_{x^+}(y^+)\ge 0.9d^F_{x^+}(y^+)$ (or any other fixed constant) without affecting the proof. Same for \eqref{eq.F-large-assumption}. The assumption that $x\in\Lambda$ is not essential. It is only used to guarantee that the $x$ starts near $E^c(\sigma)$ in the sense of \cite[Lemma 2.15, 2.16]{PYY25}. 
\end{remark}	
\begin{remark}\label{r.switch}
	By continuity and shrinking $\delta_0$ if necessary, one can {\em a priori} require that $\theta(t_x)\ge \frac12 \inf_{x\in \partial W_r^-(\sigma)}t_x$ which goes to infinity as $r\to 0$. 
	Furthermore, the assumption that $t^+>t_x$ implies $\theta(t_x) > t_x - 1$. In other words, we assume that the shadowing orbit $y$ spends at least as much time as $x$ when passing through $B_{r_0}(x)$. This assumption is not essential since once can switch $x$ and $y$ and consider the time reparametrization $\theta^{-1}$ for the orbit of $x$. Doing so does not (essentially) affect the assumption on the $E$- and $F$-length; see Corollary \ref{cor.switch-base-pf-expansivity} below. 
\end{remark}

In the following, we will prove these two lemmas under the assumption that $\sigma$ is Lorenz-like. Then the general case follows by considering $-X$, similar to \cite{PYY25a}. More details will be provided in Section \ref{ss.4.3}.

 %Therefore, we need only prove these two lemmas for the case when $\sigma$ is Lorenz-like. 
%Let us consider firstly the case when $\sigma$ is a Lorenz-like singularity. 
%Let us remark that Lemma \ref{lem.F-large} cannot be proved simply by applying Lemma \ref{lem.E-large} with the case ofto $-X$. This is because $\sigma$, as a singularity of $-X$, is not Lorenz-like. 

For $r^*>0$ small enough, we identify $B_{r^*}(\sigma)$ as the $r^*$-ball in $\mathbb{R}^{\dim M}$ centered at the origin $\sigma$. 
When $\sigma$ is Lorenz-like, we will assume for simplicity that the subspaces $E^{ss}_{\sigma}$, $E^c_{\sigma}$ and $E^u_{\sigma}$ are mutually orthogonal. Moreover, we identify these subspaces with $\mathbb{R}^{s}$, $\mathbb{R}^c$, and $\mathbb{R}^u$, respectively. Here, we write $s,c$ and $u$ as the dimensions of $E^{ss}_{\sigma}$, $E^c_{\sigma}$, $E^u_{\sigma}$, respectively. Note that $c=\dim E^c_{\sigma}=1$. 

For each $\alpha>0$, one defines on $B_{r^*}(\sigma)$ the $(\alpha,\mathbb{R}^s)$-cone, $(\alpha,\mathbb{R}^c)$-cone, and $(\alpha,\mathbb{R}^{u})$-cone, etc., in the usual way.

\begin{lemma}\label{lem.flow-direction}
	Assume that $\sigma$ is Lorenz-like. 
	For any $\alpha>0$, there exist $r_0\in (0,r^*)$ and $\bar{r}\in (0,r_0)$ such that for every $r\in (0,\bar{r}]$, the following results hold:
	\begin{itemize}
		\item for every $x\in \partial W^-_r(\sigma)\cap\Lambda$,  $X(x)$ is contained in $(\alpha/2,\mathbb{R}^c)$-cone. 
		\item for every $x\in\partial W^+_r(\sigma)\cap\Lambda$,  $X(x)$ is contained in $(\alpha/2,\mathbb{R}^u)$-cone.
	\end{itemize} 
\end{lemma}
\begin{proof}
	The first result is proved in \cite[Lemma 2.16]{PYY25} which uses a proof similar to that in \cite[Lemma 4.4]{LGW}. The second result is a direct consequence of the hyperbolicity of $\sigma$.
\end{proof}

We now assume that $0<r_0<r^*$ satisfy both Proposition \ref{prop.ff-singularity} and Lemma \ref{lem.flow-direction}, for the given $\alpha>0$. In the discussions below, we will reduce the constants $\alpha$, $r_0$ and $r^*$, when it is necessary.

\subsection{Proof of Lemma \ref{lem.E-large} when $\sigma$ is Lorenz-like}

%Let $r_0>0$, $\bar{r}\in(0,r_0)$ be given by Lemma \ref{lem.flow-direction}. 
We let \[W_r:=\bigcup_{\sigma\in\Lambda\cap\Sing(X)}W_r(\sigma)\]
which is an open neighborhood of $\Lambda\cap\Sing(X)$. Note that for a small neighborhood $U_\Lambda$ of $\Lambda$, $\overline U\setminus W_r$ is a compact set that has no singularity (we may further shrink $r_0$ and $\overline r$ such that $W_r\subset U_\Lambda$). Then, we let $\delta_0>0$ be small such that 
\[0<\delta_0<\inf_{y\in \overline U\setminus W_r}\rho_0|X(y)|.\]
This implies, in particular, that the local Poincar\'e map $\cP_{x'}$ is well-defined in $B_{\delta_0}(x')$ for every $x'\in \partial W_r(\sigma)$. Now fix $x\in(\partial W^-_r(\sigma)\setminus W^s_{loc}(\sigma))\cap\Lambda$ as in Lemma \ref{lem.E-large}.
We will write for simplicity $\cN(x)=\cN_{\delta_0}(x)\subset\cN_{\rho_0|X(x)|}(x)$ and $\cN(x^+)=\cN_{\delta_0}(x^+)\subset\cN_{\rho_0|X(x^+)|}(x^+)$. 

\subsubsection{Construction of a $cu$-graph}\label{sss.4.1.1}

Reducing $\bar{r}$, we may assume that for each $r\in (0,\bar{r}]$, the point $x\in \partial W^-_r(\sigma)$ is sufficiently close to the local stable manifold $W^s_{loc}(\sigma)$ in $B_{r^*}(\sigma)$. Consequently, 
we can take a smoothly embedded disk $D_0^u\subset \cN(x)$ of dimension $\dim D_0^u=u$ that is tangent to the $(\alpha,\mathbb{R}^u)$-cone, contains $x$ in its interior and is centered at a point $z_0\in W^s_{loc}(\sigma)$.

For each $t\ge -1$, let $D^u_t$ be the connected component of $f_t(D^u_0)\cap B_{r^*}(\sigma)$ containing $f_t(z_0)$.  
By the $\lambda$-lemma (see e.g. \cite{PdM}), there is a constant $T_0>0$ such that for any $t\ge T_0$, $D^u_t$ is $C^1$ close to the local unstable manifold of $\sigma$ in $B_{r^*}(\sigma)$, denoted as $W_{loc}^u(\sigma)$,  and it intersects transversely with $\cF^{cs}_{\sigma}(y')$ at a unique point, for every $y'\in B_{r_0}(\sigma)$. 
By further reducing $\bar{r}$ if necessary, we may assume that 
\[T:=t_x\ge T_0+1. \]

Write $D_0=f_{[-1,1]}(D^u_0)$. 
By Lemma \ref{lem.flow-direction}, we may assume that $D_0$ is tangent to the $(\alpha,\mathbb{R}^c\oplus\mathbb{R}^u)$-cone. 
For each integer $i\ge 1$, let $D_i$ be the connected component of $f(D_{i-1})\cap B_{r^*}(\sigma)$ that contains $f_i(z_0)$. By forward invariance of the $(\alpha,\mathbb{R}^c\oplus\mathbb{R}^u)$-cone field, each $D_i$ is a $(u+1)$-dimensional disk tangent to $(\alpha,\mathbb{R}^c\oplus\mathbb{R}^u)$-cone. 
Then the union 
\[\Sigma:=\bigcup_{i=0}^{\infty}D_i\]
is a $C^1$ submanifold in $B_{r^*}(\sigma)$ and is tangent to the $(\alpha,\mathbb{R}^c\oplus\mathbb{R}^u)$-cone. Moreover, the $\lambda$-lemma implies that the local unstable manifold $W^u_{loc}(\sigma)$  is contained in the boundary of $\Sigma$. 
Let \[\hat{\Sigma}:=\bigcup_{t\ge T_0}D^u_t,\] which is the part of $\Sigma$ that is close to the local unstable manifold $W^u_{loc}(\sigma)$. 

\begin{lemma}\label{lem.cu-graph}
	For each $y'\in B_{r_0}(\sigma)$, $\cF^{cs}_{\sigma}(y')$ intersects $\hat{\Sigma}$ transversely along a $C^1$ curve, which is tangent to the $(\alpha, \mathbb{R}^c)$-cone. 
\end{lemma}
\begin{proof}
	By construction, for each $y'\in B_{r_0}(\sigma)$, the leaf $\cF^{cs}_{\sigma}(y')$ intersects $D^u_t$ $(t\ge T_0)$ transversely at a unique point, which will be denoted as $y^u_t$. Then the set $\ell:=\{y^u_t\mid t\ge T_0\}$ forms a $C^1$ curve, contained in $\cF^{cs}_{\sigma}(y')\pitchfork \hat{\Sigma}$. Since $\cF^{cs}_{\sigma}(y')$ is tangent to $(\alpha, \mathbb{R}^s\oplus\mathbb{R}^c)$-cone and $\hat{\Sigma}$ is tangent to $(\alpha, \mathbb{R}^c\oplus\mathbb{R}^u)$-cone, we have that $\ell$ is tangent to the $(\alpha, \mathbb{R}^c)$-cone.
\end{proof}

\begin{figure}[htpb]
	\centering
	\includegraphics[width=0.6\textwidth]{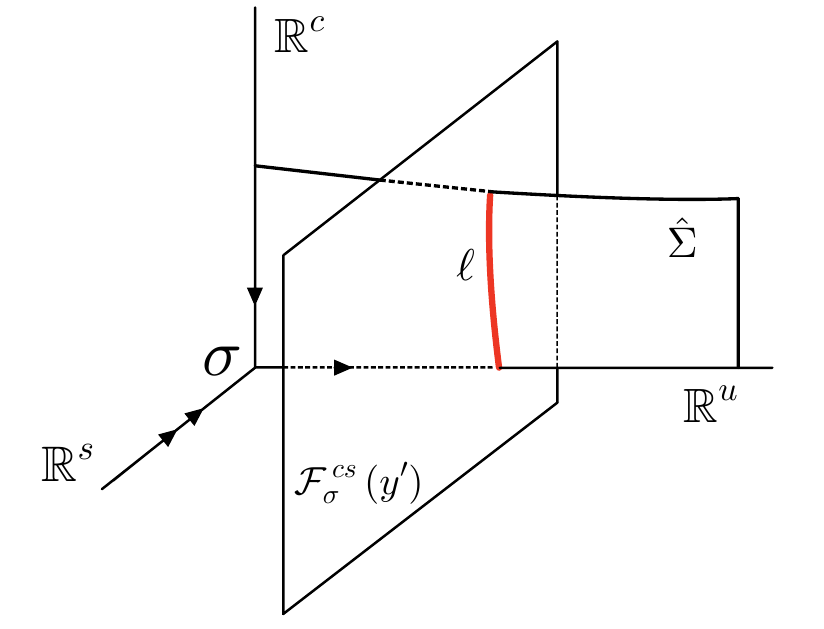}
	\caption{The intersection of $\cF^{cs}_{\sigma}(y')$ with $\hat{\Sigma}$}
	\label{fig.sigma-hat}
\end{figure}

Reducing $\alpha$ if necessary, we see that $\hat{\Sigma}$ is transverse to $\cF^{ss}_{\sigma}$ and is a $cu$-graph in the following sense: for each $y'\in B_{r_0}(\sigma)$, $\cF^{ss}_{\sigma}(y')$ intersects $\hat{\Sigma}$ transversely at most one point. 
In fact, as $\cF^{cs}_{\sigma}(y')$ is subfoliated by $\cF^{ss}_{\sigma}$ whose leaves are tangent to $(\alpha,\mathbb{R}^s)$-cone, the intersection between each leaf of $\cF^{ss}_{\sigma}$ and $\ell:=\cF^{cs}_{\sigma}(y')\pitchfork \hat{\Sigma}$ consists of at most one point. In particular, $\cF^{ss}_{\sigma}(y')$ intersects $\hat{\Sigma}$ transversely at most one point.

\subsubsection{A coordinate system near $\sigma$} \label{sss.4.1.2}

Recall that Lemma \ref{lem.E-large} assumes that $t^+\ge T=t_x$, where $t^+$ 
is taken such that 
$$
y^+ = f_{t^+}(y)=\cP_{x^+}(f_{\theta(t_x)}(y))
$$
and 
$$
|t^+ - \theta(t_x)|\le \frac{\delta_0}{\inf_{z\in U_\Lambda\setminus W_r}|X(z)|}.
$$
Since the denominator does not depend on $\delta_0$, we may shrink $\delta_0$ such that $|t^+ - \theta(t_x)|<1$ for all $x$ and $y$ satisfying the assumptions of Lemma \ref{lem.E-large}. %This further implies $\theta(t_x)> t_x -1.$

Let 
\[\hat{y}=f_{-[t^+]}(y^+),\] 
where $[t^+]$ denotes the integer closest to $t^+$ (take an arbitrary one if there are two choices). Without loss of generality, we assume also that $T$ is an integer. Then it holds $[t^+]\ge T$. 
Note that $\hat{y}=f_{\hat{t}}(y)$ for some $|\hat{t}|<1$ and $d(x,y)<\delta\in (0,\delta_0)$. By reducing  $\delta>0$ if necessary, one has $\cF^{ss}_{\sigma}(\hat{y})\pitchfork D_0\neq\emptyset$. By the condition \eqref{eq.delta-close-condition}, the orbit segment from $y$ to $y^+$ stays inside the ball $B_{r^*}(\sigma)$. 
Then it follows from local invariance of the $\cF^{ss}_{\sigma}$ foliation that 
\[\cF^{ss}_{\sigma}(\hat{y}_i)\pitchfork D_i\neq\emptyset,\quad \forall 0\le i\le [t^+],\] 
where $\hat{y}_i=f_i(\hat{y})$. In particular, $\cF^{ss}_{\sigma}(y^+)$ intersects $D_{[t^+]}$ at a unique point, which will be denoted as $y^s_T$. 
Recall that $T\ge T_0+1$. Hence, $D^u_T$ intersects $\cF^{cs}_{\sigma}(y^+)$ transversely at a point $y^u_{T}$. Denote 
\[\tilde{y}:=f_{-T}(y^+)\quad \text{and}\quad \tilde{y}_t=f_t(\tilde{y}), t = 1,2,\cdots, T-1.\]

Observe that $D^u_T$ and $D_{[t^+]}$ are both contained in $\hat{\Sigma}$. 
By Lemma \ref{lem.cu-graph}, $\cF^{cs}_{\sigma}(y^+)$ intersects $\hat{\Sigma}$ transversely at a $C^1$ curve tangent to the $(\alpha,\mathbb{R}^c)$-cone, and it contains a segment $\ell^c_T$ that connects $y^s_T$ to $y^u_T$. See Figure \ref{fig.near-x+}.
\begin{figure}[htpb]
	\centering
	\includegraphics[width=0.8\textwidth]{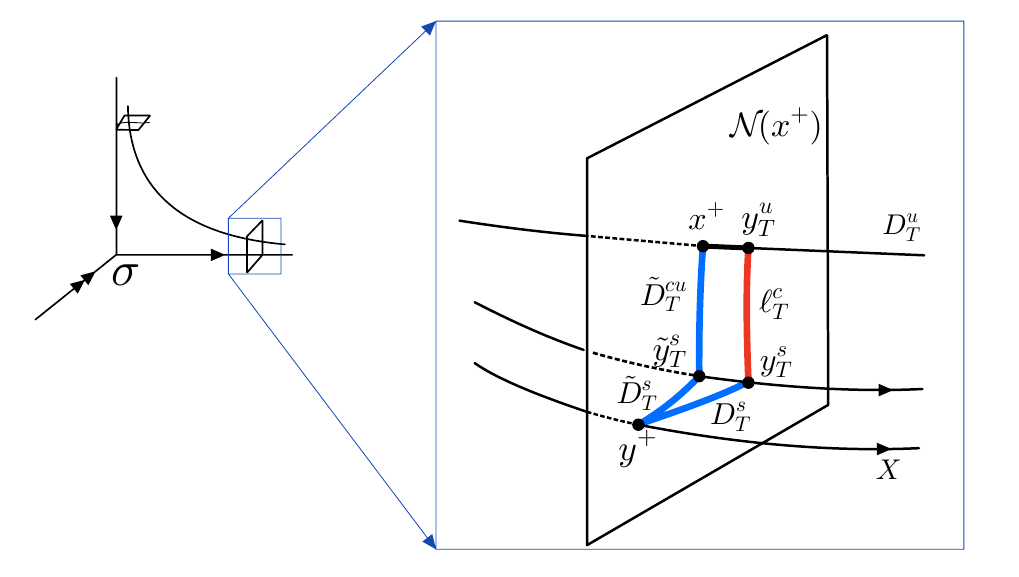}
	\caption{The triple $(D^s_T,\ell^c_T,D^u_T)$ and their projection to $\cN(x^+)$}
	\label{fig.near-x+}
\end{figure}

Define 
\[\ell^c_{T-k}=f_{-k}(\ell^c_T),%\quad D^s_{T-k}=f_{-k}(D^s_T),
\quad\forall k=1, 2, \cdots,T.\]
Then each $\ell^c_k$ $(k=0,1,\cdots, T)$ is contained in the intersection of $f_{-k}(\hat{\Sigma})$ with $\cF^{cs}_{\sigma}(\tilde{y}_k)$, where $\tilde{y}_k=f_k(\tilde{y})=f_{k-T}(y^+)$. Note that $f_{-k}(\hat{\Sigma})\subset \Sigma$, which is tangent to the $(\alpha, \mathbb{R}^c\oplus\mathbb{R}^u)$-cone, and $\cF^{cs}_{\sigma}(\tilde{y}_k)$ is tangent to the $(\alpha, \mathbb{R}^s\oplus\mathbb{R}^c)$-cone. Consequently, $\ell^c_k$ is tangent to the $(\alpha,\mathbb{R}^c)$-cone. Moreover, denote for each $k=0,1,\cdots, T$ the points
\[y^{\iota}_{k}:=f_{k-T}(y^{\iota}_{T}),\quad \iota=s,u,\]
which are the endpoints of $\ell^c_k$. 
Then $y^s_{k}$ is the unique intersection of $\cF^{ss}_{\sigma}(\tilde{y}_k)$ with $f_{-k}(\hat{\Sigma})$, and $y^u_{k}$ is the unique intersection of $D^u_k$ with $\cF^{cs}_{\sigma}(\tilde{y}_k)$. We can now take an $s$-dimensional embedded disk $D^s_0\subset\cF^{ss}_{\sigma}(\tilde{y})$ that is centered at $\tilde{y}$ and contains $y^s_0$. Let $D^s_k=f_k(D^s_0)$ for $k=1,2,\cdots, T$. Then for each $k=0,1,\cdots, T$,  $D^s_k\subset \cF^{ss}_{\sigma}(\tilde{y}_k)$ is an embedded disk containing $\tilde{y}_k$ and $y^s_k$. 

In this way, we have defined a triple $(D^s_k, \ell^c_k, D^u_k)$ for each $k=0,1,\cdots, T$, that connects $\tilde{y}_k$ to $x_k$, via the points $y^s_k$ and $y^u_k$. Below, we show how to use the triples to prove the assertions in Lemma \ref{lem.E-large}. 

\subsubsection{Transition from $(\cF^{cs}_{x^+,\cN},\cF^{cu}_{x^+,\cN})$ to $(D^s_T,\ell^c_T, D^u_T)$}\label{sss.4.1.3}

This step is exactly the same as in \cite[Proof of Lemma 6.7, Step 1]{PYY25}. We include the details here for completeness. 

Near the point $x^+$, the triple $(D^s_T,\ell^c_T,D^u_T)$ may not be contained in the normal manifold $\cN(x^+)$.  However, $D^s_T$ is almost parallel to $\cN(x^+)$, as it is tangent to the $(\alpha,\mathbb{R}^s)$-cone while the flow direction is almost parallel to $\RR^u$. Let 
\[\tilde{D}^s_T=\cP_{x^+}(D^s_T),\quad \tilde{y}^s_T=\cP_{x^+}(y^s_T).\]
We see that $\tilde{D}^s_T$ is tangent to the $(2\alpha,N^{cs}(x^+))$-cone, reducing $\delta\in (0,\delta_0)$ if necessary. 
Note that $\ell^c_T$ and $D^u_T$ are contained in $\hat{\Sigma}$, which is a $(u+1)$-dimensional disk tangent to the $(\alpha,\mathbb{R}^c\oplus\mathbb{R}^u)$-cone and consists of entire flow segments. Moreover, near $x^+$, the flow segments are almost orthogonal to $\cN(x^+)$. Therefore, the intersection of $\hat{\Sigma}$ with $\cN(x^+)$ contains a $u$-dimensional disk tangent to the $(2\alpha,N^{cu}(x^+))$-cone, which can be extended to a disk $\tilde{D}^{cu}_T$ tangent also to $(2\alpha,N^{cu}(x^+))$-cone and it contains the points $x^+$, $\tilde{y}^s_T$.  See Figure \ref{fig.near-x+}.

Now, on $\cN(x^+)$ we have the following objects:
\begin{itemize}
	\item $s$-dimensional disks $\tilde{D}^s_T$ and $\cF^{cs}_{x^+,\cN}(y^+)$, which are both tangent to the $(2\alpha, N^{cs}(x^+))$-cone and contain $y^+$;
	\item $u$-dimensional disks $\tilde{D}^{cu}_T$ and $\cF^{cu}_{x^+,\cN}(x^+)$, which are both tangent to the $(2\alpha,N^{cu}(x^+))$-cone and contain $x^+$;
	\item $\tilde{D}^s_T\pitchfork \tilde{D}^{cu}_T=\{\tilde{y}^s_T\}$, and $\cF^{cs}_{x^+,\cN}(y^+)\pitchfork \cF^{cu}_{x^+,\cN}(x^+)=\{y^{+,F}\}$. 
\end{itemize}

Since we have assumed large $E$-length at $x^+$, see \eqref{eq.E-large-assumption}, 
by reducing $\alpha>0$ if necessary, \cite[Lemma 6.4]{PYY25} shows that  
\begin{equation}\label{eq.E-length-control-T}
	d_{\tilde{D}^s_T}(\tilde{y}^s_T,y^+)\in (0.9,1.1)d_{\cF^{cs}_{x^+,\cN}}(y^{+,F},y^+)=(0.9,1.1)d_{x^+}^E(y^+).
\end{equation}
Since $\tilde{D}^s_T$ and $D^s_T$ are both close to $x^+$ and tangent to the $(3\alpha,\mathbb{R}^s)$-cone, one has
$d_{D^s_T}(y^s_T,y^+)$ is arbitrarily close to $d_{\tilde{D}^s_T}(\tilde{y}^s_T,y^+)$ (reducing $\alpha$ if necessary), hence by \eqref{eq.E-length-control-T}, 
\begin{equation}\label{eq.E-lenth-substitute}
	d_{D^s_T}(y^s_T,y^+)\ge \frac{1}{2}d_{x^+}^E(y^+).
\end{equation}
Note that by the assumption of large $E$-length \eqref{eq.E-large-assumption}, there exists a constant $c_1>0$ which depends only on the integers $s, u$ and the dimension of $\cN(x^+)$, such that 
\[d_{x^+}^E(y^+)\ge c_1d_{\cN(x^+)}(x^+,y^+).\]
It follows from \eqref{eq.E-lenth-substitute} that 
\[d_{D^s_T}(y^s_T,y^+)\ge\frac{1}{2}c_1d_{\cN(x^+)}(x^+,y^+)\ge\frac{1}{2}c_1d(x^+,y^+).\]
As elements of the triple $(D^s_T, \ell^c_T, D^u_T)$ are pairwise almost orthogonal with their dimensions add up to $\dim M$, there exists $c_2>0$ depending only on the integers $s,u$ and $\dim M$, such that 
\begin{align*}
d(x^+,y^+)&\ge c_2\left(d_{D^u_T}(x^+,y^u_T)+\mathrm{length}(\ell^c_T)+d_{D^s_T}(y^s_T,y^+)\right)\\
&\ge c_2\max\{\mathrm{length}(\ell^c_T),d_{D^u_T}(x^+,y^u_T)\}.
\end{align*}
Therefore, we have 
\begin{equation}\label{eq.s-large-at-T}
	d_{D^s_T}(y^s_T,y^+)\ge c_3\max\{\mathrm{length}(\ell^c_T),d_{D^u_T}(x^+,y^u_T)\},
\end{equation}
where $c_3=c_1c_2/2$.

\subsubsection{Backward expansion of $E$-length}\label{sss.4.1.4}

%In this step, we show backward expansion of the $s$-distance between $y^s_i$ and $\tilde{y}_i$, from $i=T$ to $i=0$. 
%This is also the same as in [PYY]. The conclusion of this step is the following: 
By \eqref{eq.s-large-at-T} and the domination $(E^{ss}_{\sigma}\oplus E^c_{\sigma})\oplus_{<} E^u_{\sigma}$, there are constants $c_4>0$ and $\tilde{\lambda}_{\sigma}>1$ such that 
\begin{equation}\label{eq.back-expansion-s-large}
	d_{D^s_{T-i}}(y^s_{T-i},\tilde{y}_{T-i})\ge c_3c_4\tilde{\lambda}_{\sigma}^i\max\{\mathrm{length}(\ell^c_{T-i}),d_{D^u_{T-i}}(x_{T-i},y^u_{T-i})\},
\end{equation}
for $i=0,1,\cdots,T$. 

Moreover, since vectors in the $(\alpha,\mathbb{R}^s)$-cone are exponentially expanded along backward orbits, we obtain
\begin{equation}\label{eq.back-expansion-T-steps}
	d_{D^s_0}(y^s_0,\tilde{y}_0)\ge c_4\tilde{\lambda}_{\sigma}^T d_{D^s_T}(y^s_T,y^+)\ge \frac{1}{2}c_4\tilde{\lambda}_{\sigma}^T d^E_{x^+}(y^+),
\end{equation}
where the second inequality follows from \eqref{eq.E-lenth-substitute}.
%More precise estimates may relate $\tilde{\lambda}_{\sigma}$ to (the inverse of) the largest exponent of $Df_1|_{E^s_{\sigma}}$. 

\subsubsection{Transition back to $(\cF^{cs}_{x,\cN},\cF^{cu}_{x,\cN})$} \label{sss.4.1.5}
Observe that $\tilde{y}=f_{-T}(y^+)$ may not be close to $x$ (this happens when $t^+\gg T = t_x$). 
Let us denote $m=[t^+]-T\ge 0$ and 
\[\tilde{y}_{-k}=f_{-k}(\tilde{y})=f_{-T-k}(y^+),\quad  \forall k=0, 1,\cdots, m.\] 
Note that $\tilde{y}_{-m}=f_{-[t^+]}(y^+)=\hat{y}$ and for the backward iterates from $\tilde{y}_0$ to $\tilde{y}_{-m}$, the $E$-length will be further expanded exponentially. We will use the exponential expansion of $E$-length to show that it remains large at $x$. See Figure \ref{fig.transition-back-to-x} for an illustration.
\begin{figure}[htpb]
	\centering
	\includegraphics[width=.8\textwidth]{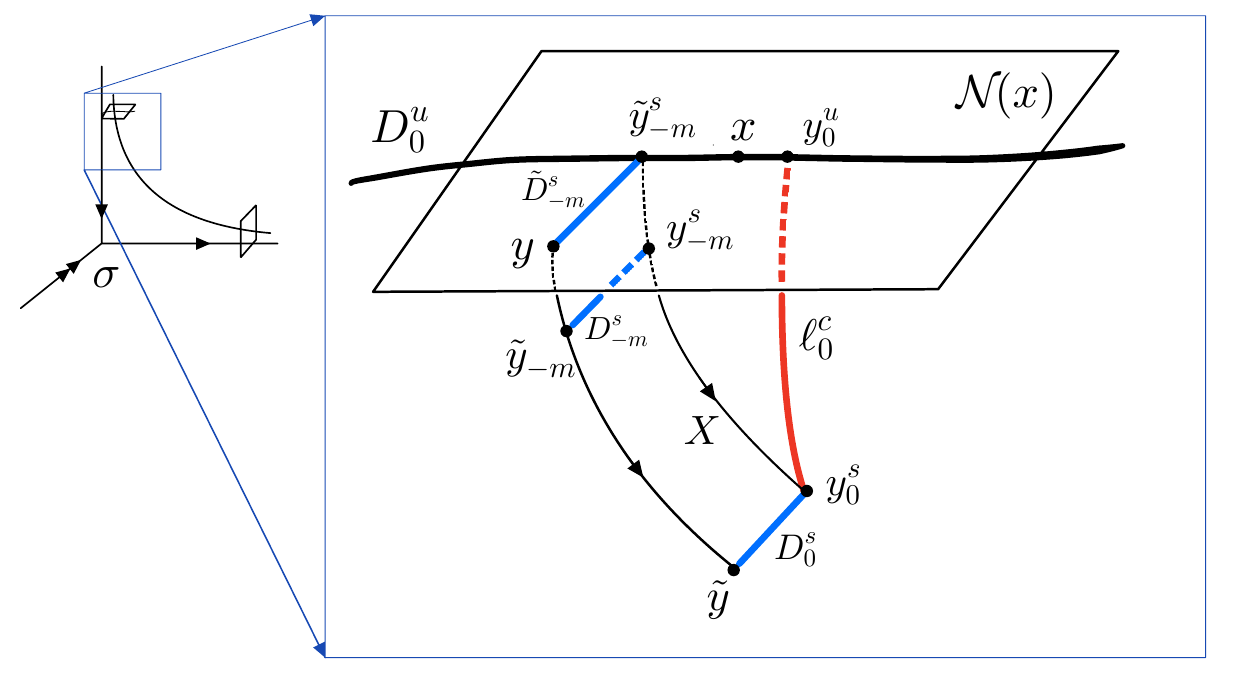}
	\caption{The triple $(D_0^s,\ell^c_0,D^u_0)$ and backward iterates}
	\label{fig.transition-back-to-x}
\end{figure}

If $m=0$, it is obtained in \cite[Proof of Lemma 6.7, Step 3]{PYY25} that one has large $E$-length at $x$. More precisely, it holds $d_x^E(y)\ge d^F_x(y)$. The backward expansion on $E$-length follows from \eqref{eq.back-expansion-T-steps}.

In the following, we assume that $m\ge 1$. 
For each $k=1,\cdots, m$, let $D^s_{-k}=f_{-k}(D^s_0)$. Then $D^s_{-k}$ contains $\tilde{y}_{-k}$ and $y^s_{-k}=f_{-k}(y^s_0)$. Since the orbit segment of $y$ stays close to that of $x$, we may assume without loss of generality that every $D^s_{-k}$ stays inside $B_{r^*}(\sigma)$. By local invariance of the $\cF^{ss}_{\sigma}$ foliation, $D^s_{-k}$ is a $C^1$ disk tangent to the $(\alpha,\mathbb{R}^s)$-cone. 

Let us consider firstly the length of $\ell^c_0$. 
By construction, $\ell^c_T$ joins the points $y^u_T\in D^u_T$ and $y^s_T\in D_{[t^+]}$. In particular, $\ell^c_T$  intersects $D^u_{T+1}$. It follows that $\ell^c_0=f_{-T}(\ell^c_T)$ contains a subsegment that joins $D^u_0$ and $D^u_1$. Since $D^u_0\subset \cN(x)$ and is almost orthogonal to the flow, there exists $\vep_0>0$, independent of $x$, such that the distance between $D^u_0$ and $D^u_1$ is at least $\vep_0|X(x)|$. It follows that 
\[\mathrm{length}(\ell^c_0)\ge\vep_0|X(x)|.\]  
By \eqref{eq.back-expansion-s-large} we have
	\begin{equation*}
		d_{D^s_0}(y^s_0,\tilde{y}_0)\ge c_3c_4\tilde{\lambda}_{\sigma}^T\mathrm{length}(\ell^c_0)\ge \vep_0 c_3c_4\tilde{\lambda}^T_{\sigma} |X(x)|.
	\end{equation*}
	Then the backward expansion of vectors in $(\alpha,\mathbb{R}^s)$-cone implies that 
	\begin{equation}\label{eq.further-back-expansion-s}
		d_{D^s_{-m}}(y^s_{-m},\tilde{y}_{-m})\ge \vep_0 c_3c_4\tilde{\lambda}^{T+m}_{\sigma} |X(x)|.
	\end{equation}
%It seems that, with this inequality and as $T$ large, we can obtain a contradiction to the fact that the orbit of $y$ stays close to $x$ and hence contained in $B_{r^*}(\sigma)$. 

%{To Rusong: I think this already contradicts the assumption (once projected to $\cN(x)$) that $y\in \cN_{\rho_0|X(x)|}(x)$. This makes sense because if $E$ is large at $x$ then $\sigma$ being of Lorenz type implies that the orbit of $y$ must stay in the tubular neighborhood of $x$ (PYY25 already uses this fact), and therefore the shadowing time can be well-controlled. As a result, $m\gg 1$ (which was used to obtain $\length(\ell^c_0) \ge \vep_0|X(x)|$) is impossible in this case. Does this simplify the proof? }

Note that $\tilde{y}_{-m}=\hat{y}$ is close to $y$ and $x$. We can project $D^s_{-m}$ (which is almost parallel to $\cN(x)$) along the flow lines to $\cN(x)$ via the holonomy map, obtaining an $s$-dimensional disk $\tilde{D}^s_{-m}$ tangent to $(2\alpha, N^{cs}(x))$-cone, together with a point $\tilde{y}^s_{-m}$ which is the projection of $y^s_{-m}$. Observe that the projection of $\tilde{y}_{-m}$ to $\cN(x)$ is precisely $y$, and $\tilde{y}^s_{-m}\in D^u_0$.  Therefore, we have 
\begin{equation}\label{eq.s-length-compare-x}
	d_{\tilde{D}_{-m}^s}(\tilde{y}^s_{-m},y)\in(1-c_{\delta}\alpha, 1+c_{\delta}\alpha)d_{D^s_{-m}}(y^s_{-m},\tilde{y}_{-m}),
\end{equation}
where $c_{\delta}>0$ is a constant that remains bounded as $\delta\to 0$. 

On $\cN(x)$, we now have the following objects:
\begin{itemize}
	\item $\tilde{D}^s_{-m}$ and $\cF^{cs}_{x,\cN}(y)$: they are both tangent to the $(2\alpha, N^{cs}(x))$-cone and contains $y$;
	\item $D^u_0$ and $\cF^{cu}_{x,\cN}(x)$: they are both tangent to the $(2\alpha,N^{cu}(x))$-cone and contains $x$;
	\item $\tilde{D}^s_{-m}\pitchfork D^u_0=\{\tilde{y}^s_{-m}\}$ and $\cF^s_{x,\cN}(y)\pitchfork \cF^{cu}_{x,\cN}(x)=\{y^F\}$.
\end{itemize}
Note that $N^{cs}(x)$ and $N^{cu}(x)$ are almost orthogonal to each other, we have that 
\begin{equation}\label{eq.dist-x-y}
	d_{\cN(x)}(x,y)\ge (1-c'_{\delta}\alpha)\max\{d_x^E(y), d_x^F(y)\},
\end{equation}
where $c'_{\delta}>0$ is some constant that remains bounded as $\delta\to 0$. Moreover, since $y\in \cN(x)\subset\cN_{\rho_0|X(x)|}(x)$, we have 
\[d_{\cN(x)}(x,y)\le \rho_0|X(x)|.\]
From this relation and \eqref{eq.further-back-expansion-s}, \eqref{eq.dist-x-y}, we obtain
\begin{align}
	d_{D^s_{-m}}(y^s_{-m},\tilde{y}_{-m})&\ge \rho_0^{-1}\vep_0 c_3c_4\tilde{\lambda}^{T+m}_{\sigma} d_{\cN_x}(x,y) \notag\\
	&\ge \rho_0^{-1}\vep_0 c_3c_4\tilde{\lambda}^{T+m}_{\sigma} (1-c'_{\delta}\alpha)\max\{d_x^E(y), d_x^F(y)\}.\label{eq.s-length-rel}
\end{align}
We now claim that 
\[d_{\tilde{D}^s_{-m}}(\tilde{y}^s_{-m},y)\ge d_{D^u_0}(\tilde{y}^{s}_{-m},x),\] 
by reducing $\alpha$ and $\bar{r}$ if necessary. If not, applying \cite[Lemma 6.4]{PYY25} we would have 
\[d_{x}^F(y)\in (0.9,1.1)d_{D^u_0}(\tilde{y}^{s}_{-m},x).\]
Hence by \eqref{eq.s-length-compare-x}, \eqref{eq.s-length-rel},
\begin{align*}
	d_{\tilde{D}_{-m}^s}(\tilde{y}^s_{-m},y)&\ge (1-c_{\delta}\alpha)d_{D^s_{-m}}(y^s_{-m},\tilde{y}_{-m})\\
	&\ge (1-c_{\delta}\alpha)\rho_0^{-1}\vep_0 c_3c_4\tilde{\lambda}^{T+m}_{\sigma} (1-c'_{\delta}\alpha) d_x^F(y)\\
	&\ge \frac{10}{9}(1-c_{\delta}\alpha)\rho_0^{-1}\vep_0 c_3c_4\tilde{\lambda}^{T+m}_{\sigma} (1-c'_{\delta}\alpha) d_{D^u_0}(\tilde{y}^{s}_{-m},x).
\end{align*}
As $T>0$ can be large enough (by reducing $\bar{r}$), we obtain a contradiction. This proves the claim. 

By the claim and applying \cite[Lemma 6.4]{PYY25}, we can further reducing $\alpha>0$ if necessary such that
\begin{equation}\label{eq.E-length-s-length-rel}
	d_x^E(y)\in (0.9,1.1)d_{\tilde{D}^s_{-m}}(\tilde{y}^s_{-m},y).
\end{equation}
It follows from this relation and \eqref{eq.s-length-compare-x},   \eqref{eq.s-length-rel} that 
\begin{align*}
	d_x^E(y)&\ge 0.9(1-c_{\delta}\alpha)d_{D^s_{-m}}(y^s_{-m},\tilde{y}_{-m})\\
	&\ge 0.9(1-c_{\delta}\alpha)\rho_0^{-1}\vep_0 c_3c_4\tilde{\lambda}^{T+m}_{\sigma} (1-c'_{\delta}\alpha)d_x^F(y)\\
	&\ge d_x^F(y),
\end{align*}
as $T$ is large enough. This proves the large $E$-length at $x$.

To see backward exponential expansion of $E$-length, note firstly that backward expansion of vectors in $(\alpha,\mathbb{R}^s)$-cone implies  
\[d_{D^s_{-m}}(y^s_{-m},\tilde{y}_{-m})\ge d_{D^s_0}(y^s_0,\tilde{y}_0).\]
Then by \eqref{eq.E-length-s-length-rel}, \eqref{eq.s-length-compare-x}, and \eqref{eq.back-expansion-T-steps}, we have 
\begin{align*}
	d_x^E(y) &\ge \frac{1}{4}d_{D^s_{-m}}(y^s_{-m},\tilde{y}_{-m})\ge  \frac{1}{4}d_{D^s_0}(y^s_0,\tilde{y}_0)\\
	&\ge  \frac{1}{8}c_4\tilde{\lambda}_{\sigma}^T d^E_{x^+}(y^+).
\end{align*} 
As $T$ can be arbitrarily large by letting $0<r\le \bar{r}\to 0$, we can take $\lambda_{\sigma}\in (1,\tilde{\lambda}_{\sigma})$ such that 
\[d_x^E(y)\ge \lambda_{\sigma}^Td^E_{x^+}(y^+).\]
This completes the proof of Lemma \ref{lem.E-large} in the case when $\sigma$ is Lorenz-like.

\subsection{Proof of Lemma \ref{lem.F-large} when $\sigma$ is Lorenz-like}
First we note that all the construction in Sections \ref{sss.4.1.1} and \ref{sss.4.1.2} still apply. This means that the objects $(D^s_t,\ell^c_t,D^u_t)$ as well as their reference points $y^\iota_t$, $\tilde y_t$ are still well-defined.

Next, observe that we must have $d^F_{x^+}(y^+)\ge d^E_{x^+}(y^+)$ (i.e. large $F$-length at $x^+$). If otherwise, we are in the settings of Lemma \ref{lem.E-large} and it would follow from the same argument that $d^E_{x}(y)>d^F_x(y)$ (in fact exponentially large), a contradiction to the assumption of Lemma \ref{lem.F-large}. 
Hence, we only need to show forward exponential expansion of the $F$-length. %We define the same objects and use the same notations as in the proof of Lemma \ref{lem.E-large}. 

\subsubsection{A first implication of large $F$-length at $x$}
 
We can project the objects $D_{-m}^s$, $y^s_{-m}$ via the local Poincar\'e map $\cP_x$ to the normal manifold $\cN(x)$, obtaining $\tilde{D}_{-m}^s$ and $\tilde{y}^s_{-m}$ as in Section \ref{sss.4.1.5}. 
Applying \cite[Lemma 6.4]{PYY25}, the large $F$-length condition \eqref{eq.F-large-assumption} implies that 
\begin{equation}\label{eq.large-F-x}
	d_{D^u_0}(x,\tilde{y}_{-m}^s)\ge\max\left\{ 0.9 d_x^F(y), \frac{1}{2}d_{\tilde{D}_{-m}^s}(\tilde{y}_{-m}^s,y)\right\}.
\end{equation}

\subsubsection{Forward expansion}

Following from the uniform expansion along $(\alpha,\mathbb{R}^u)$-cone,  there are $c_1>0$ and $\tilde{\lambda}_{\sigma}>1$ such that 
\begin{equation}\label{eq.u-length-expansion}
	d_{D^u_i}(x_i,y^u_i)\ge c_1\tilde{\lambda}^i_{\sigma}d_{D^u_0}(x,y^u_0),\quad i=0,1,\cdots,T.
\end{equation}

Now, we project $\ell^c_0$ to $\cN(x)$ along the flow lines, obtaining a curve $\tilde{\ell}^c_0\subset D^u_0$. Note that $\tilde{\ell}^c_0$ joins $\tilde{y}_{-m}^s$ and $y^u_0$. Moreover, since the time from a point on $\ell^c_0$ to $\cN(x)$ is bounded from above, the curve $\ell^c_0$ is $C^1$. 
We also project $\ell^c_T$ to $\cN(x^+)$ along the flow lines, obtaining a $C^1$ curve $\tilde{\ell}^c_T$. Since $\ell^c_T$ is tangent to the $(\alpha,\mathbb{R}^c)$-cone and almost parallel to $\cN(x^+)$, by reducing $\alpha$ and $\delta$ if necessary, we may assume that 
\begin{equation}\label{eq.c-length-T}
 2\mathrm{length}(\ell^c_T) \ge	\mathrm{length}(\tilde{\ell}^c_T)\ge \frac{1}{2}\mathrm{length}(\ell^c_T).
\end{equation}
Note that the curve $\tilde{\ell}^c_T$ is contained in the disk $D_T$ which is tangent to the $(\alpha,\mathbb{R}^c\oplus \mathbb{R}^u)$-cone. 
By the sectional hyperbolicity near the Lorenz-like singularity $\sigma$, the area of parallelogram tangent to $(\alpha,\mathbb{R}^c\oplus \mathbb{R}^u)$-cone is forward expanding. This implies that, decreasing $c_1$ and $\tilde{\lambda}_{\sigma}$ if necessary, 
\begin{equation}\label{eq.sect-exp-c}
	\mathrm{length}(\tilde{\ell}^c_T))|X(x^+)|\ge c_1\tilde{\lambda}^T_{\sigma}\mathrm{length}(\tilde{\ell}^c_0)|X(x)|.
\end{equation}
Since the flow speed $|X(\cdot)|$ depends Lipschitz continuously on $d(\cdot, \sigma)$, and $x, x^+$ are both contained in $\partial B_{r_0}(\sigma)$, we have 
\[|X(x)|/|X(x^+)|>c_2,\]
for some constant $c_2>0$. It follows from \eqref{eq.c-length-T} and \eqref{eq.sect-exp-c} that 
\begin{equation}\label{eq.c-length-expansion}
	\mathrm{length}(\ell^c_T)\ge \frac{1}{2}\mathrm{length}(\tilde{\ell}^c_T)\ge \frac{1}{2}c_1c_2\tilde{\lambda}^T_{\sigma}\mathrm{length}(\tilde{\ell}^c_0).
\end{equation}

Observe that inside $D^u_0$ we have 
\[d_{D^u_0}(x,y^u_{0})+\mathrm{length}(\tilde{\ell}^c_0)\ge d_{D^u_0}(x,\tilde{y}^s_{-m}).\]
Then by \eqref{eq.u-length-expansion} and \eqref{eq.c-length-expansion}, 
\begin{align}
	d_{D^u_T}(x_T,y^u_T)+\mathrm{length}(\ell^c_T) &\ge c_3\tilde{\lambda}_{\sigma}^T \big(d_{D^u_0}(x,y^u_0)+\mathrm{length}(\tilde\ell^c_0)\big)\notag\\
	&\ge c_3\tilde{\lambda}_{\sigma}^T d_{D^u_0}(x,\tilde{y}^s_{-m}), \label{eq.c_3}
\end{align}
where $c_3=\min\{c_1, c_1c_2/2\}$.

\subsubsection{Transition from $(D^s_T,\ell^c_T,D^u_T)$ to $(\cF^{cs}_{x^+,\cN},\cF^{cu}_{x^+,\cN})$}
We refer the reads to Figure \ref{fig.near-x+} and \ref{fig.transition-back-to-x} for the objects. 
As is shown that on $\cN(x^+)$, we have $d^F_{x^+}(y^+)\ge d^E_{x^+}(y^+)$. It follows from \cite[Lemma 6.4]{PYY25} that 
\begin{align}
   d_{\tilde{D}_T^{cu}}(x^+,\tilde{y}^s_T)\in (0.9, 1.1)d^F_{x^+}(y^+), \quad\text{and} \notag\\
	d_{\tilde{D}_T^{cu}}(x^+,\tilde{y}^s_T)\ge \frac{1}{2} d_{\tilde{D}^s_T}(\tilde{y}_T^s,y^+). \label{eq.F-large-consequence}
\end{align} 
Since $\tilde{D}^{cu}_T$ and $\tilde{D}^s_T$ are tangent to the $(2\alpha,N^{cu}(x^+))$-cone and $(2\alpha,N^{cs}(x^+))$-cone, respectively, there exists a constant $c_4>0$ depending only on the integers $s$ and $u$ such that
\begin{equation}\label{eq.c_4}
	d_{\tilde{D}_T^{cu}}(x^+,\tilde{y}^s_T)\ge c_4d_{\cN(x^+)}(x^+,y^+)\ge c_4 d(x^+,y^+).
\end{equation}
Since $D^u_T$, $\ell^c_T$ and $D^s_T$ are almost pairwise orthogonal, there exists a constant $c_5>0$ depending only on $s, u$ and $\dim M$ such that
\[d(x^+,y^+)\ge c_5 \max\left\{d_{D^u_T}(x^+,y^u_T),\mathrm{length}(\ell^c_T), d_{D^s_T}(y^s_T,y^+)\right\}.\]
Then,  there exists $c_6>0$ depending only on $s, u$ and $\dim M$ such that
\begin{equation}\label{eq.c_6}
	d(x^+,y^+)\ge c_6\left(d_{D^u_T}(x^+,y^u_T)+\mathrm{length}(\ell^c_T)\right).
\end{equation}
In total, we have 
\begin{align*}
	d^F_{x^+}(y^+) 
	&\ge \frac{1}{2}d_{\tilde{D}^{cu}_T}(x^+,\tilde{y}^s_T) %\ge \frac{1}{4}d_{\tilde{D}^s_T}(\tilde{y}^s_T,y^+) 
	&& 
	\text{by \eqref{eq.F-large-consequence}}\\
	&\ge \frac{1}{4}c_4d(x^+,y^+) &&\text{by \eqref{eq.c_4}}\\
	&\ge \frac{1}{4}c_4c_6(d_{D^u_T}(x^+,y^u_T)+\mathrm{length}(\ell^c_T)) &&\text{by \eqref{eq.c_6}}\\
	&\ge \frac{1}{4}c_3c_4c_6\tilde{\lambda}^T_{\sigma}d_{D^u_0}(x,\tilde{y}^s_{-m}) &&\text{by \eqref{eq.c_3}}\\
	&\ge \frac{1}{8}c_3c_4c_6\tilde{\lambda}^T_{\sigma}d_x^F(y). &&\text{by \eqref{eq.large-F-x}}
\end{align*}
Since the constants does not depend on $r$ or $T$, we can fix $\lambda_{\sigma}\in (1,\tilde{\lambda}_{\sigma})$ and reduce $\bar{r}$ (therefore increase $T$), such that $\frac{1}{8}c_3c_4c_6\tilde{\lambda}^T_{\sigma}\ge \lambda_{\sigma}^T$. This concludes the proof of Lemma \ref{lem.F-large} in the case that $\sigma$ is Lorenz-like.

\subsection{Completing the proof of Lemma \ref{lem.E-large} and Lemma \ref{lem.F-large}}\label{ss.4.3}

We are now left to prove Lemma \ref{lem.E-large} and Lemma \ref{lem.F-large} for the case when $\sigma$ is reversed Lorenz-like. By definition, $\sigma$ is reversed Lorenz-like if it is Lorenz-like for the reversed flow $-X$. 
Observe that the fake foliations constructed for $X$ remain fake foliations for $-X$, and they give also a coordinate system along regular orbits, only that the $E$-length and $F$-length are exchanged to each other. 

To be clear, let us denote $Y=-X$, and suppose $\sigma$ is reversed Lorenz-like for $X$. In other words, $\sigma$ is Lorenz-like for $Y$. If we have large $E$-length at $x^+$ (for $X$), then we have large $F$-length at $x^+$ for $Y$. Observe that, for $Y$, the orbit of $x$ approaches $\sigma$ from $x^+$ and leaves the neighborhood of $\sigma$ at $x$. Thus, when $\sigma$ is reversed Lorenz-like (for $X$), one proves Lemma \ref{lem.E-large}  by simplying applying Lemma \ref{lem.F-large} to $Y$ for the singularity $\sigma$ (which is Lorenz-like for $Y$). Similarly, when $\sigma$ is reversed Lorenz-like (for $X$), one proves Lemma \ref{lem.F-large} by applying Lemma \ref{lem.E-large} to $Y$. 

This finishes the proof of Lemma \ref{lem.E-large} and Lemma \ref{lem.F-large} for all cases.

\begin{remark}\label{r.par}
	We finish this section by noting that all the parameters in Lemma \ref{lem.E-large} and \ref{lem.F-large} can be taken uniformly in a $C^1$ small neighborhood of the vector field $X$. Indeed, the constants $\overline r_0, \overline r$, $\delta_0$ $\sigma_\sigma$ all depend on the hyperbolicity of the singularities, which is robust under $C^1$ perturbation.  
\end{remark}

\section{Expansivity: proof of Theorem \ref{m.e}}\label{s.5}
{In this section we prove Theorem \ref{m.e} for the vector field $X$ itself, under the extra assumption that $\Lambda$ is isolated, i.e., it is the maximal invariant set of an open neighborhood $U_\Lambda\supset \Lambda$. If this is not true, then just take a sufficiently small neighborhood $U_\Lambda$ of $\Lambda$ and replace $\Lambda$ by $\Lambda_X$, the maximal invariant set of $X$ in $U_\Lambda$. By \cite{CLYZ}, if $U_\Lambda$ is sufficiently small, then $\Lambda_X$ is also multi-singular hyperbolic with the same singularities as $\Lambda$, and our proof below applies to $\Lambda_X$ in a verbatim way. Since being multi-singular hyperbolic is a $C^1$ open property (see \cite{CLYZ}), the same proof also applies to all vector fields $Y$ in a $C^1$ small neighborhood of $X$, and it is straightforward to check the uniformity of $\delta_K$  using the uniformity of constants in \ref{lem.E-large} and \ref{lem.F-large} for nearby vector fields; see Remark \ref{r.par}.
	 
}

Given a vector field $X\in\xX^1(M)$, let $\rho_0>0, K_0>1$ be constants given by Lemma \ref{lem.liao-size}. 
Let $\Lambda$ be a multi-singular hyperbolic set of $X$, which admits a dominated splitting of normal bundle $N_{\Lambda}=N^{cs}\oplus N^{cu}$. 
As in Section \ref{sect.coord-reg}, we take $\beta>0$ and assume $\rho_0$ to be small enough so that \eqref{eq.assumption-on-rho0} holds, and the leaves of the fake foliations $\cF^{cs}_{x,\cN}$, $\cF^{cu}_{x,\cN}$ given by Proposition \ref{prop.ff-regular} are tangent to the $\beta$-cones around the $N^{cs}$, $N^{cu}$ bundles, respectively. Below, we will further reduce $\beta$ and $\rho_0$.

Let $\alpha>0$ and $r^*>r_0>0$ be small enough such that Proposition \ref{prop.ff-singularity} holds for every $\sigma\in\Lambda\cap\Sing(X)$. 
Reducing $r_0$ if necessary, there exists $\bar{r}\in (0,r_0)$, $\delta_0>0$ such that for any $r\in (0,\bar{r})$ and $\delta\in (0,\delta_0)$, the conclusions in Lemma \ref{lem.E-large} and \ref{lem.F-large} hold. 
Let us denote 
\[V:=\bigcup_{\sigma\in\Lambda\cap\Sing(X)}\overline{B_{r_0}(\sigma)}.\]
Assume $r_0$ is small enough such that $V$ is a compact isolating neighborhood of $\Lambda\cap\Sing(X)$. 
By multi-singular hyperbolicity of $\Lambda$, there exist constants $\eta, T_V>0$ such that 
\begin{equation}\label{eq.muti-sing-hyp-prime}
			\|\psi_t|_{N^{cs}(x)}\|<e^{-\eta t}, \quad \|\psi_{-t}|_{N^{cu}(f_t(x))}\|<e^{-\eta t},
\end{equation}
whenever $x,f_t(x)\in\Lambda\setminus V$ and $t\ge T_V$;

By Proposition \ref{prop.ff-regular} and uniform continuity of $D\cP^*_{1,x}$ (see \cite[Lemma 2.5]{GY}), we can reduce $\beta$ and $\rho_0$ such that for any $x\in\Lambda\setminus\Sing(X)$ and $y\in\cN_{\rho_0|X(x)|}(x)$, it holds
\begin{gather}
	d^F_{x_1}(\cP_{1,x}(y))\ge e^{-\eta/10}\cdot m(\psi_1|_{N^{cu}(x)})\cdot d^F_x(y),\label{eq.F-expansion} \\
\quad d^E_{x_{-1}}(\cP_{1,x_{-1}}^{-1}(y))\ge e^{-\eta/10} \cdot m(\psi_{-1}|_{N^{cs}(x)})\cdot d^E_x(y),\label{eq.E-contraction}
\end{gather}
where $x_{\pm 1}=f_{\pm 1}(x)$ as before. 

Given $r\in (0,\bar{r})$, we denote, as in Lemma \ref{lem.E-large} and \ref{lem.F-large},
\[W_r:=\bigcup_{\sigma\in\Lambda\cap\Sing(X)}W_r(\sigma)\subset V,\]
where $W_r(\sigma)$ is defined in the beginning of Section \ref{sect.est-sing}. Note that each $W_r(\sigma)$ depends also on the constant $r_0$. 
By reducing $r_0$ and $r$ if necessary, for any $\sigma\in\Lambda\cap\Sing(X)$ and $x\in \partial W^-_r(\sigma)\cap W_{loc}^s(\sigma)$, it holds that $W^s_{loc}(\sigma)\cap \cN_{\rho_0|X(x)|}(x)$ is a $C^1$ disk tangent to the cone $C^{cs}_{\beta}$ on $\cN_{\rho_0|X(x)|}(x)$. Assuming $\beta$ to be small enough, any point on $W^s_{loc}(\sigma)\cap \cN_{\rho_0|X(x)|}(x)$ must have large $E$-length, i.e.
\begin{equation}\label{eq.E-large-upon-entering}
	d_x^E(\tilde{y})> d_x^F(\tilde{y}),\quad\forall \tilde{y}\in W^s_{loc}(\sigma)\cap \cN_{\rho_0|X(x)|}(x)\setminus\{x\}.
\end{equation}
Similarly, for any $x\in \partial W^+_r(\sigma)\cap W^u_{loc}(\sigma)$, any point on $W^u_{loc}(\sigma)\cap\cN_{\rho_0|X(x)|}(x)$ must have large $F$-length, i.e.
\begin{equation}\label{eq.F-large-upon-leaving}
	d_x^F(\tilde{y})> d_x^E(\tilde{y}),\quad\forall \tilde{y}\in W^u_{loc}(\sigma)\cap \cN_{\rho_0|X(x)|}(x)\setminus\{x\}.
\end{equation} 

The next lemma and corollary will be used to switch $x$ and $y$ when necessary (see Remark \ref{r.switch})
\begin{lemma}\label{lem.switch-base-pf-expansivity}
	By reducing $\alpha$ and $\delta_0$ if necessary, for any $\delta\in (0,\delta_0)$ and $x\in\partial W^-_r\cup\partial W^+_r$, $y\in \cN_{\rho_0|X(x)|}(x)$, suppose $d(x,y)<\delta$, 
\begin{itemize}
\item if $d_{\cF^{cu}_{x,\cN}}(y^E,y)\ge 0.99 d_{\cF^{cs}_{x,\cN}}(x,y^E)$, then 
\[d_y^F(\tilde{x}) = d_{\cF^{cu}_{y,\cN}}(y,\tilde{x}^F) \ge  0.9d_{\cF^{cs}_{y,\cN}}(\tilde{x}^F, \tilde{x}) = 0.9 d_y^E(\tilde{x}),\]
where $\tilde{x}=\cP_y(x)\in \cN_{\rho_0|X(y)|}(y)$. 
\item similarly, if $d_{\cF^{cs}_{x,\cN}}(x,y^E)\ge 0.99d_{\cF^{cu}_{x,\cN}}(y^E,y)$, then 
\[ d_y^E(\tilde{x}) = d_{\cF^{cs}_{y,\cN}}(\tilde x^F,\tilde x) \ge  0.9d_{\cF^{cu}_{y,\cN}}(y,\tilde{x}^F) = 0.9 d_y^F(\tilde{x}).\]
\end{itemize}
\end{lemma}
\begin{proof}
    Let us define 
    $\tilde{\cF}^{cu}_{y,\cN}(y):=\cP_y(\cF^{cu}_{x,\cN}(y))$ and $\tilde{\cF}^{cs}_{y,\cN}(\tilde{x}):=\cP_y(\cF^{cs}_{x,\cN}(x))$. 
	By reducing $\delta_0$ if necessary, we have that $\cN_{\rho_0|X(y)|}(y)$ are almost parallel to $\cN_{\rho_0|X(x)|}(x)$ such that 
	\begin{itemize}
		\item $\tilde{\cF}^{cu}_{y,\cN}(y), \tilde{\cF}^{cs}_{y,\cN}(\tilde{x})$ are tangent the cones $C^{cu}_{2\alpha}$ and $C^{cs}_{2\alpha}$, respectively;
		\item moreover, $d_{\tilde{\cF}^{cu}_{y,\cN}(y)}(y,\tilde{y}^E)\ge 0.98d_{\tilde{\cF}^{cs}_{y,\cN}(\tilde{x})}(\tilde{y}^E,\tilde{x})$, where $\tilde{y}^E=\cP_y(y^E)$.
	\end{itemize}
	Then by \cite[Lemma 6.4]{PYY25} and reducing $\alpha$ if necessary, one obtains 
	\[d_{\cF^{cu}_{y,\cN}}(y,\tilde{x}^F)\ge 0.9d_{\cF^{cs}_{y,\cN}}(\tilde{x}^F, \tilde{x}),\]
	where $\tilde{x}^F$ is the unique intersection of $\cF^{cu}_{y,\cN}(y)$ with $\cF^{cs}_{y,\cN}(\tilde{x})$. By our notations of $E$-length and $F$-length, this is exactly the inequality for the first item. The second item can be proved similarly. 
\end{proof}
Combining Lemma \ref{lem.switch-base-pf-expansivity} and Corollary \ref{cor.switch-base} we obtain the following result. 
\begin{corollary}\label{cor.switch-base-pf-expansivity}
	By reducing $\alpha$ and $\delta_0$ if necessary, for any $\delta\in (0,\delta_0)$ and $x\in\partial W^-_r\cup\partial W^+_r$, $y\in \cN_{\rho_0|X(x)|}(x)$, suppose $d(x,y)<\delta$, 
	\begin{itemize}
		\item if $d_x^F(y)\ge d_x^E(y)$, then $d_y^F(\tilde{x})\ge 0.9 d^E_y(\tilde{x})$;
		\item if $d_x^E(y)\ge d_x^F(y)$, then $d_y^E(\tilde{x})\ge 0.9 d^F_y(\tilde{x})$.
	\end{itemize}
\end{corollary}

By reducing $r$ while keeping $r_0$ fixed, we can require that for any $x\in \partial W^-_r(\sigma)$, suppose $t_x>0$ is the smallest positive constant that $f_t(x)\in \partial W^+_r(\sigma)$, then 
\begin{equation*}
	0.9^2\lambda_{\sigma}^{T/2}C_0\ge 1,
\end{equation*}
where
$$
C_0 : = \sup_{0<t<T_V,z,f_t(z)\notin V}\|\psi_{-t}|_{N^{cu}(f_t(z))}\|,
$$
$\lambda_{\sigma}>1$ is given by Lemma \ref{lem.E-large} and \ref{lem.F-large}, and $T_V>0$ is the constant given by multi-singular hyperbolicity of $\Lambda$ such that \eqref{eq.muti-sing-hyp-prime} holds whenever $t>T_V$. 
And recall that we have 
\begin{equation*}
	\rho_0|X(z)|>\delta_0,\quad \forall z\in U_\Lambda\setminus W_r
\end{equation*}
where $U_\Lambda$ is a small neighborhood of $\Lambda$.

Now, for any $\delta\in (0,\delta_0)$ and for any $x,y\in \Lambda$, suppose 
\begin{equation}\label{eq.delta-close}
	d(f_{\theta(t)}(y),f_t(x))\le \delta,\quad\forall t\in \mathbb{R},
\end{equation} 
where $\theta: \mathbb{R}\to \mathbb{R}$ is a time-reparametrization. 
%If $x\in W^s(\sigma)$ for some singularity $\sigma$, then it follows from hyperbolicity of the singularity and \eqref{eq.delta-close} that $y\in W^s(\sigma)$. Similarly, if $x\in W^u(\sigma)$, then $y\in W^u(\sigma)$. 
%In particular, if $x$ is a singularity, we must have $y=x$. 
We will assume in the following lemma that neither $x$ nor $y$ is a singularity. 
By taking iterates of the points, we may assume that $x\notin\Int(V)$.  
Moreover, to prove Theorem \ref{m.e}, we may assume $\theta(0)=0$, i.e. $d(y,x)\le \delta$, and show that $y$ is contained in $f_{[-\vep,\vep]}(x)$ for some $\vep>0$, which can be arbitrarily small as $\delta\to 0$.\footnote{Note that this conclusion may not hold if $x\in V$. In that case, by taking iteration, we can conclude that there exists $t_0$ for which $f_{\theta(t_0)}(y)\in f_{[t_0-\vep,t_+\vep]}(x)$.}

For the time being we will consider the case 
\[d^F_x(\tilde{y})\ge d^E_x(\tilde{y}),\] 
where  $\tilde{y}=\cP_x(y)$.

We consider the forward orbits of $x$ and $y$. 
Suppose the forward orbit $f_t(x)$ of $x$ enters $W_r$ for the first time at $t=t_1^-$ and then leaves $W_r$ at time $t=t_1^+$. Denote 
\[x_1^-=f_{t_1^-}(x)\in\partial W^-_r,\quad x_1^+=f_{t_1^+}(x)\in\partial W^+_r, \]
and 
\[\tilde{y}^-_1=\cP_{x^-_1}(f_{\theta(t^-_1)}(y)),\quad \tilde{y}^+_1=\cP_{x^+_1}(f_{\theta(t^+_1)}(y)).\]
$\tilde y_1^-$ and $\tilde y_1^+$ are well-defined due to the choice of $\delta_0.$
Let $\Delta s_1=t_1^+-t_1^-$, $\Delta\tilde{s}_1=\inf\{s>0: f_s(\tilde{y}_1^-)=\tilde{y}_1^+\}$, and $\tau_1=\min\{\Delta s_1,\Delta \tilde{s}_1\}$. 
Denote $\Delta t_0=t_1^-$.

\begin{lemma}\label{lem.first-circle}
Assume $d^F_x(\tilde{y})\ge d^E_x(\tilde{y})$ where $\tilde{y}=\cP_x(y)$.	If $\Delta t_0\le T_V$, then 
\[d^F_{x^+_1}(\tilde{y}^+_1)\ge \lambda_{\sigma_1}^{\tau_1/2}d^F_x(y);\]
if $\Delta t_0>T_V$, then 
\[d^F_{x^+_1}(\tilde{y}^+_1)\ge \lambda_{\sigma_1}^{\tau_1/2} e^{0.9\eta \Delta t_0} d^F_x(y).\]

\end{lemma}

\begin{proof}
The domination of the splitting $N_{\Lambda}=N^{cs}\oplus N^{cu}$ implies   
\begin{equation}\label{eq.F-large-pf-expansivity}
	d^F_{x^-_1}(\tilde{y}^-_1)\ge d^E_{x^-_1}(\tilde{y}^-_1).
\end{equation}
If $\Delta s_1\le \Delta \tilde{s}_1$, we can apply Lemma \ref{lem.F-large} to obtain that 
\[d^F_{x^+_1}(\tilde{y}^+_1)\ge d^E_{x^+_1}(\tilde{y}^+_1),\quad\text{and}\quad d^F_{x^+_1}(\tilde{y}^+_1)\ge \lambda_{\sigma_1}^{\Delta s_1} d^F_{x^-_1}(\tilde{y}^-_1).\]
If $\Delta s_1>\Delta \tilde{s}_1$, Corollary \ref{cor.switch-base-pf-expansivity} gives  
\[d_{\tilde{y}^-_1}^F(\tilde{x}^-_1)\ge 0.9d_{\tilde{y}^-_1}^E(\tilde{x}^-_1),\]
where $\tilde{x}^-_1=\cP_{\tilde{y}^-_1}(x^-_1)$. 
Then we can apply Lemma \ref{lem.F-large} by using $\tilde{y}^-_1$ as the reference point (keeping in mind Remark \ref{r.EF} and \ref{r.switch}), obtaining 
\[d_{\tilde{y}^+_1}^F(\tilde{x}^+_1)\ge d_{\tilde{y}^+_1}^E(\tilde{x}^+_1),\quad \text{and}\quad d_{\tilde{y}^+_1}^F(\tilde{x}^+_1)\ge \lambda_{\sigma_1}^{\Delta \tilde{s}_1}d_{\tilde{y}^-_1}^F(\tilde{x}^-_1),\]
where $\tilde{x}^+_1=\cP_{\tilde{y}^+_1}(x^+_1)$. 
Consequently, by Corollary \ref{cor.switch-base-pf-expansivity} we have
\[d^F_{x^+_1}(\tilde{y}^+_1)\ge 0.9 d^E_{x^+_1}(\tilde{y}^+_1),\]
and
\[d^F_{x^+_1}(\tilde{y}^+_1)\ge 0.9 d_{\tilde{y}^+_1}^F(\tilde{x}^+_1)\ge 0.9\lambda_{\sigma_1}^{\Delta \tilde{s}_1}d_{\tilde{y}^-_1}^F(\tilde{x}^-_1)\ge 0.9^2\lambda_{\sigma_1}^{\Delta \tilde{s}_1} d^F_{x^-_1}(\tilde{y}^-_1).\]

For the orbit segment from $x$ to $x^-_1$, it follows from \eqref{eq.F-expansion} that 
\begin{equation}\label{eq.t0}
	d^F_{x^-_1}(\tilde{y}^-_1)\ge e^{-\eta\Delta t_0/10}\|\psi_{\Delta t_0}|_{N^{cu}(x)}\| d^F_{x}(\tilde{y}).
\end{equation}
If $\Delta t_0\le T$, we have 
\begin{align*}
	d^F_{x^+_1}(\tilde{y}^+_1)&\ge 0.9^2\lambda_{\sigma_1}^{\tau_1} d^F_{x^-_1}(\tilde{y}^-_1)\\
	&\ge 0.9^2\lambda_{\sigma_1}^{\tau_1}C_0 d^F_x(y)\\
	&\ge \lambda_{\sigma_1}^{\tau_1/2}d^F_x(y).
\end{align*}
If $\Delta t_0>T$, then it follows from \eqref{eq.muti-sing-hyp-prime} and \eqref{eq.F-expansion} that 
\begin{align*}
	d^F_{x^+_1}(\tilde{y}^+_1)&\ge 0.9^2\lambda_{\sigma_1}^{\tau_1} d^F_{x^-_1}(\tilde{y}^-_1)\\
	&\ge 0.9^2\lambda_{\sigma_1}^{\tau_1} e^{-\eta\Delta t_0/10}\|\psi_{\Delta t_0}\|d^F_x(y)\\
	&\ge 0.9^2\lambda_{\sigma_1}^{\tau_1} e^{0.9\eta \Delta t_0} d^F_x(y)\\
	&\ge \lambda_{\sigma_1}^{\tau_1/2} e^{0.9\eta \Delta t_0} d^F_x(y),
\end{align*}
where the last inequality in both cases follows as we assume $\tau_1$ to be large enough. 
This completes the proof of the claim. 
\end{proof}

\begin{proof}[Proof of Theorem \ref{m.e}]
Let $x,y\in \Lambda$ be arbitrary, satisfying $d(f_t(x), f_{\theta(t)}(y))<\delta$ for some time reparametrization (i.e., an orientation-preserving homeomorphism of $\RR$) $\theta$. Further assume, by replacing $x,y$ with other points in the same orbit if necessary, that $y\in B_\delta(x)$ and $\theta(0) = 0$.

Define the set $\Lambda^\dagger = \Lambda\setminus \big(W^s(\Sing(X)\cap\Lambda)\cup W^u(\Sing(X)\cap\Lambda)\big)$.
We consider the following four cases.

 \medskip
\noindent {\em Case 1.} Either $x$ or $y$ is a singularity. In this case, the smallness of $\delta$ (for which $B_{\delta}(\Sing(X)\cap\Lambda)$ is an isolating neighborhood of $\Sing(X)\cap\Lambda$) and the surjectivity of $\theta$ guarantees that $y=x.$ 

\medskip 
\noindent {\em Case 2.} $x,y$ are both regular, and $x\in \Lambda^\dagger$. 
In this case we assume, without loss of generality, that \[d^F_x(\tilde{y})\ge d^E_x(\tilde{y}),\] 
where  $\tilde{y}=\cP_x(y)$. The other case can be proven along the same lines by considering $-X$.

The goal is to show that $d^F_x(\tilde{y})=0$ by considering the forward orbit of $x$. There are two subcases.

\smallskip
\noindent {\em Subcase 2.1.} 
Suppose the forward orbit of $x$ crosses $W_r$ infinitely many times. 
Let $0<t^-_1<t^+_1<t^-_2<t^+_2<\cdots<t^-_k<t^+_k<\cdots$ be the sequence of times such that 
\begin{itemize}
	\item $f_t(x)\in \Lambda\setminus W_r$, for all $t\in [0,t^-_1]$, or $t\in [t^+_k,t^-_{k+1}]$ and $k\in\mathbb{N}_+$;
	\item $f_t(x)\in W_r(\sigma_k)$, for all $t\in(t^-_k,t^+_k)$ and $k\in\mathbb{N}_+$, where $\sigma_k\in\Lambda\cap\Sing(X)$.
	%\item $f_t(x)\in \Lambda\setminus W_r$, for all $t\in (t^+_k,t^-_{k+1})$ and $k\in\mathbb{N}_+$.
\end{itemize}
Let us denote 
\[x_0=x,\quad\text{and}\quad x_k^-=f_{t^-_k}(x),\quad x_k^+=f_{t^+_k}(x),\quad \forall k\in\mathbb{N}_+,\]
and 
\[\Delta t_0=t^-_1, \quad  \Delta s_k=t^+_k-t^-_k,\quad \Delta t_k=t^-_{k+1}-t^+_k,\quad\forall k\in\mathbb{N}_+.\] 
Let $\tilde{y}^-_k=\cP_{x^-_k}(f_{\theta(t^-_k)}(y))$, $\tilde{y}^+_k=\cP_{x^+_k}(f_{\theta(t^+_k)}(y))$ and 
\[\Delta\tilde{s}_k=\inf\{s>0: f_s(\tilde{y}^-_k)=\tilde{y}^+_k\}.\] 
Denote $\Delta_k=\min\{\Delta s_k,\Delta\tilde{s}_k\}$. 
Applying Lemma \ref{lem.first-circle} inductively, one has 
\[d^F_{x^+_k}(\tilde{y}^+_k)\ge \tilde{\lambda}^{S'_k/2}d^F_x(\tilde{y}),\]
where $S'_k=\Delta_1+\cdots+\Delta_k$, and $\tilde{\lambda}>1$ is the minimal of all the constants $\lambda_{\sigma}$. 
Since $d^F_{x^+_k}(\tilde{y}^+_k)$ is uniformly bounded above, and $S'_k\to \infty$ as $k\to\infty$, we must have $d^F_x(\tilde{y})=0$.

\smallskip
\noindent {\em Subcase 2.2.}
Now, suppose the forward orbit of $x$ crosses $W_r$ only $K$ times at 
$0<t_1^-<t_1^+<\cdots<t_K^-<t_K^+$, and after that, it never return to $W_r$. %In particular, $\omega(x)$ is a hyperbolic set (Remark \ref{rmk.msh-nonsingular}). 
In this case, we apply Lemma \ref{lem.first-circle} inductively for the first $K$ times, obtaining
\[d^F_{x^+_K}(\tilde{y}^+_K)\ge \tilde{\lambda}^{S'_K/2}d^F_x(\tilde{y}).\] 
Denote $x'=x^+_K$, $x'_t=f_t(x')=f_{t_K^++t}(x)$, and $\tilde{y}'_t=\cP_{x'_t}(f_{\theta(t_K^++t)}(y))$. If $K = 0$ we just let $x' = x$ and $\tilde y_t' = \cP_{x_t}(f_{\theta(t)}(y))$.
Then by multi-singular hyperbolicity \eqref{eq.muti-sing-hyp-prime} and \eqref{eq.F-expansion}, for each integer $n\ge T$ for which $x_n'\notin V$, we have 
\begin{equation*}
	d^F_{x'_{n}}(\tilde{y}'_{n})\ge e^{-n\eta/10}\|\psi_{n}|_{N^{cu}(x')}\|d^F_{x_K^-}(\tilde{y}^-_K)\ge e^{0.9n\eta}d^F_{x_K^-}(\tilde{y}^-_K).
\end{equation*}
There are infinitely many such $n$'s since $V$ is an isolating neighborhood of $\Sing(X)\cap\Lambda$, and therefore the forward orbit of $x$ and $y$ cannot stay in $V$ (unless either of them is in $W^s(\Sing(X)\cap\Lambda)$, but this is the next case). Along such a sequence of $n$, $d^F_{x'_{n}}(\tilde{y}'_{n})$ remains uniformly bounded, and we must have $d^F_x(\tilde{y})=0$.

We have proven in both cases that if $d^F_x(\tilde{y})\ge d^E_x(\tilde{y})$, then 
$d^F_x(\tilde{y})=0$. This forces $d^E_x(\tilde{y})=0$, and so $\tilde{y}=x$.

\medskip \noindent {\em Case 3.}
$x,y$ are both regular, and $x\in W^s(\Sing(X)\cap\Lambda)\Delta W^u(\Sing(X)\cap\Lambda$ where $A\Delta B$ denotes the symmetric difference of $A$ and $B$, i.e., $A\Delta B = (A\setminus B) \cup (B\setminus A).$ We will only consider the case $x\in W^s(\Sing(X)\cap\Lambda)$ but not in $W^u(\Sing(X)\cap\Lambda)$, and the other case follows by considering $-X$.

Now, suppose $x\in W^s(\sigma)\setminus\{\sigma\}$ for some singularity $\sigma$, then the condition \eqref{eq.delta-close} implies that $y\in W^s(\sigma)\setminus\{\sigma\}$. Considering forward iterates, we may assume that $x\in \partial W^-_r(\sigma)\cap W^s_{loc}(\sigma)$ and $\tilde{y}=\cP_x(y)\in W^s_{loc}(\sigma)\cap \cN_{\rho_0|X(x)|}(x)$. 
Then it follows from \eqref{eq.E-large-upon-entering} that $d^E_x(\tilde{y})\ge d^F_x(\tilde{y})$. We can thus consider the backward orbits of $x$ and $y$. Since $x$ is not contained in the unstable manifold of any singularity, the same proof as in Case 2 shows that $\tilde{y}=x$. 

\medskip \noindent {\em Case 4.}
$x,y$ are both regular, and $x\in W^s(\Sing(X)\cap\Lambda)\cap W^u(\Sing(X)\cap\Lambda$. Say $x\in W^s(\sigma)\cap W^u(\sigma')$ for some $\sigma,\sigma'\in\Sing(X\cap \Lambda)$.

In this case we may simplify take $x\in W^u_{loc}(\sigma')\cap \partial W^+_r(\sigma')$ and $\tilde{y}=\cP_{x}(y)\in \cN_{\rho_0|X(x)|}(x)$. As before, the shadowing property forces $\tilde y\in W^u_{loc}(\sigma')$ and consequently  $d_{x}^F(\tilde{y})\ge d_{x}^E(\tilde{y})$.
Also denote by $\tau$ the last time the orbit of $x$ enters $W_r$. 
The same proof as in Case 2 shows that 
$d_{f_\tau(x)}^F(\cP_{f_\tau(x)}(f_{\theta(\tau)}(y)))\gg d_{f_\tau(x)}^E(\cP_{f_\tau(x)}(f_{\theta(\tau)}(y)))$,
i.e., the $F$-length gets much larger when compared to the $E$-length upon entering $W_r(\sigma)$ for the last time. 

However, since $\tau$ is the last entrance of the forward orbit of $x$ to $W_r$, we must have $f_{\tau}(x)\in W^s_{loc}(\sigma)$, and the the shadowing between the orbits of $x$ and $y$ forces $f_{\theta(\tau)}(y)\in W^s_{loc}(\sigma)$. In particular, this requires $d_{f_\tau(x)}^F(\cP_{f_\tau(x)}(f_{\theta(\tau)}(y)))\le d_{f_\tau(x)}^E(\cP_{f_\tau(x)}(f_{\theta(\tau)}(y)))$,
which is a contradiction unless $d_{x}^F(\tilde{y})= d_{x}^E(\tilde{y})=0$, that is, $\tilde y = x$ 

\medskip 

\noindent {\em Wrapping up the proof.}
In all cases we have proven that $\cP_x(y) = \tilde y = x,$ and so $y = f_s(x)$ for some small $s$. Since $x$ is taken outside of an isolating neighborhood of $\Sing(X)\cap \Lambda$, the flow speed is bounded away from zero, and we must have $s\to 0$ as $\delta \to 0$. This shows that $X|_\Lambda$ is Komuro expansive, finishing the proof of Theorem \ref{m.e}.

\end{proof}

{
It is worth noting that in Case 2 we do not need $\theta$ to be surjective. Indeed, whenever the orbit of $x$ crosses $W_r$, the orbit of $y$ must then spend a uniform amount of time in $V$, i.e., $\Delta\tilde s_k > a>0$ \footnote{For instance, one can take $\displaystyle a = \frac{2(r_0-r)}{\sup_{x\in M}|X(x)|}$.} for some constant $a>0$. Then if the forward orbit of $x$ visits $V$ infinitely many times, then $\theta(t)$ must go to infinity as $t\to\infty$. On the other hand, if the forward orbit of $x$ only visit $W_r$ a finite number of times, then so must $y$, and we also have $\theta(t)\to\infty$ in this case. The same argument applies to $\lim_{t\to\infty}\theta(t)$ by considering the backward orbit. We hence establish the following, slightly stronger version of Komuro expansivity for points in $\Lambda^\dagger.$

\begin{theorem}\label{m.strongerK}
	Suppose $\Lambda$ is a multi-singular hyperbolic set of $(f_t)_{t\in\RR}$ and $\Lambda^\dagger$ is obtained by removing the stable and unstable manifolds of all singularities from $\Lambda$. Then, for all $\vep>0$ there exists $\delta>0$ such that for all $x\in\Lambda^\dagger$ and $y\in\Lambda$, and all continuous, strictly increasing function $\theta:\RR\to\RR$, we have
	$$
	d(f_t(x),f_{\theta(t)}(y))<\delta,\forall t\in\RR \implies f_{\theta(t_0)}(y)\in f_{[t_0-\vep,t_0+\vep]}(x) \mbox{ for some $t_0\in\RR$}.
	$$ 
\end{theorem}
In certain sources this is known as K-expansivity and the Komuro expansivity in this paper is called $K^*$-expansivity. 

Theorem \ref{m.strongerK} cannot be further improved. Indeed, if $x$ belongs to the stable manifold of some $\sigma\in\Lambda$ then one can take some $y$, not in the orbit of $x$, together with a strictly increasing, continuous, yet bounded from above function $\theta:\RR\to\RR$ such that, as $t\to\infty$, $f_{\theta(t)}(y)$ goes to some $z\in B_\delta(x)$ as $f_t(z)\to\sigma$. Same can be said if $x$ belongs to the unstable manifold of some $\sigma\in\Lambda$. Indeed, Oka \cite{Oka} proved that $K$-expansivity at all points is equivalent to BW-expansivity, which forbids singularities to be approached by regular orbits. 

It is also worth noting, using the Poincar\'e recurrence theorem, that $\Lambda^\dagger$ has full mass for any invariant, ergodic probability measure on $\Lambda$ that is not a point mass of a singularity. In view of the almost expansivity defined in \cite{PYY21,PYY25a} We call this property the {\em almost $K$-expansivity}.
}

\section{Counting and equidistribution of periodic orbits: proof of Theorem \ref{m.counting} and \ref{m.robust.counting}}
\label{sec.counting-equidistribution}

Throughout this section, let $\Lambda$ be an isolated multi-singular
hyperbolic chain recurrence class and write
\[
h=h_{\mathrm{\mathrm{top}}}(X|_\Lambda)>0.
\]
We assume that $\Lambda$ satisfies the assumptions of Theorems \ref{m.counting}. In particular, this places $\Lambda$ under the assumptions of \cite[Theorem G]{PYY25a} and consequently there exists a unique, ergodic measure of maximal entropy $\mu = \mu_{MME}$. 
Furthermore, the ``good'' collection $\cG$ and the corresponding fixed-scale partition-sum estimates all hold (see \cite[Section 3 and 6.1]{PYY25a}; see also \cite{PYYY}).
We record explicitly the following estimate from \cite{PYY25a} and \cite{PYYY} for constant potential function $\varphi \equiv 0$. 

%
%We recall that orbit segments $(x,t)$ in $\cG$ satisfies the following:
%\begin{itemize}
%	\item $x$ is a forward hyperbolic time and recurrence Pliss time to 
%\end{itemize}

For $t>0$, set
\[
d_t(x,y)=\max_{0\le s\le t}d(f_s(x),f_s(y)).
\]
Given $Z\subset \Lambda$ and $\epsilon>0$, let
\[
\Lambda(Z,\epsilon,t)
=\sup\{\#E:E\subset Z\text{ is }(t,\epsilon)\text{-separated}\}.
\]
For a collection $\cC\subset \Lambda\times[0,\infty)$ of orbit segments,
we write $\cC_t=\{x:(x,t)\in\cC\}$ and
$\Lambda(\cC,\epsilon,t)=\Lambda(\cC_t,\epsilon,t)$.

\begin{proposition}	\label{prop.counting-inputs}\cite[Lemma 3.15 and Proposition 5.5]{PYYY}
	There is $\epsilon_G>0$ \footnote{For the choice of $\epsilon_G$, see \cite[Section 6.2]{PYY25a}; roughly speaking, it can be taken as $\frac{\vep}{100L_X}$ where $\vep$ is the scale of expansivity and $L_X$ is the Lipschitz constant for the time-one map $f_1$.} and constants $Q=Q(\epsilon_G)>1$,
	%$\tau=\tau(\epsilon_G)>0$, 
	and $n_0\ge1$ with the following properties for
	all integers $n\ge n_0$:
	\begin{enumerate}
		\item the global partition sum has the uniform upper bound
		\begin{equation}\label{eq.global-partition-upper}
			\Lambda(\Lambda,\epsilon_G,n)\le Qe^{hn};
		\end{equation}
		\item the following Gibbs upper bound holds: there exists $n_1>0$ and $Q'>0$ such that for all $n>n_1$ and any orbit segment $(x,n)$, it holds
		\begin{equation}\label{eq.gibbs-G}
			%(Q')^{-1}e^{-hn}\le 
			\mu(B_n(x,\epsilon_G)) \le Q'e^{-hn}.
		\end{equation}

%		\item for every $(x,n)\in\cG$, there is a periodic point $p=p(x,n)\in
%		\Lambda$ and a return time $q=q(x,n)>0$ such that
%		\begin{equation}\label{eq.periodic-closing}
%			f_q(p)=p,\qquad q\le n+\tau,
%			\qquad d_n(x,p)<\epsilon_G.
%		\end{equation}
	\end{enumerate}
	%The return time $q$ in \eqref{eq.periodic-closing} is not required to be theprime period of $p$, and no lower estimate of the form $q\ge n-o(1)$ is 	used below.
\end{proposition}

\subsection{The lower bound}
In this section we will establish the lower bound estimate of Theorem \ref{m.counting}, namely 	$\#\Pi_\Lambda(t)\ge C^{-1}\frac{e^{ht}}{t}$ for some constant $C>1$, for all $t$ sufficiently large. This is summarized as the following proposition.

\begin{proposition}\label{prop.cumulative-lower}
	Under the assumptions of Theorem \ref{m.counting}, there is a constant $c>0$ such that
	\[
	\#\Pi_\Lambda(t)\ge c\frac{e^{ht}}{t}
	\]
	for every sufficiently large $t$.
\end{proposition}

The proof of Proposition \ref{prop.cumulative-lower} requires the following lemma.

%\begin{lemma}[Upper Gibbs estimate and the lower partition sum]
%	\label{lem.upper-Gibbs-K-partition-lower}
%	For any $n\in\NN$ and any compact set 
%	$K\subset\Lambda$ with $\mu(K)>0$, 
%	%	for every sufficiently large
%	%	integer $n$, and
%	%	\begin{equation}\label{eq.upper-Gibbs-on-K}
%		%		\mu\bigl(B_n(x,\epsilon_{ G})\bigr)
%		%		\le C_{ G}e^{-hn}
%		%		\quad\forall x\in K
%		%		\text{ and every sufficiently large }n.
%		%	\end{equation}
%	%	Then, for every sufficiently large $n$,
%	it holds that 
%	\begin{equation}\label{eq.K-partition-lower}
%		\Lambda(K,\eta,n)
%		\ge \frac{\mu(K)}{Q'}e^{hn}
%	\end{equation}
%	for all $\eta<\epsilon_G$.
%	
%	%	In particular,
%	%	\[
%	%	\Lambda(\cG,\epsilon,n)
%	%	\ge \Lambda(K,\epsilon,n)
%	%	\ge \frac{\mu(K)}{Q'}e^{hn}.
%	%	\]
%\end{lemma}
%\begin{proof}
%	Let $E_n\subset K$ be maximal $(n,\eta)$-separated.  Maximality
%	implies that the closed Bowen balls
%	$\overline B_n(x,\epsilon)$, $x\in E_n$, cover $K$.  
%	%	We thus obtain
%	%	\[
%	%	\overline B_n(x,4\epsilon)
%	%	\subset B_n(x,\epsilon_{ G})
%	%	\quad\text{for every }x\in E_n.
%	%	\]
%	Consequently, Proposition \ref{prop.counting-inputs} gives
%	\[
%	\mu(K)
%	\le \sum_{x\in E_n}
%	\mu\bigl(\overline B_n(x,\eta)\bigr)
%	\le Q'e^{-hn}\#E_n.
%	\]
%	Thus
%	\[
%	\#E_n\ge \frac{\mu(K)}{Q'}e^{hn},
%	\]
%	as desired. 
%	%Since $K\subset\cG_n$, the corresponding lower bound for $\Lambda(\cG,4\epsilon,n)$ follows.
%\end{proof}

\begin{lemma}
	\label{lem.recurrent-hyperbolic-closing-block}
	For any $\epsilon>0$ sufficiently small, there exist a compact set
	\[
	K=K(\epsilon)\subset\Lambda\setminus\operatorname{Sing}(X)
	\]
	with $ \mu(K)>0$, and constants $\kappa>0$ and $T_0>0$ with the following property.  Whenever
	$T\ge T_0$ and
	\[
	x,f_T(x)\in K,
	\]
	there exist a periodic point $p\in\Lambda$ and a return time $q>0$ such that
	\begin{equation*}\label{eq.recurrent-block-closing}
		f_q(p)=p,
		\qquad |q-T|\le\kappa,
		\qquad d_T(x,p)<\epsilon,
	\end{equation*}
	where $d_T$ is the Bowen $T$-metric.
\end{lemma}
We remark that $q$ is not required to be the prime period of $p$.

\begin{proof}
	The unique MME $\mu$ is ergodic and hyperbolic (see \cite{CLYZ}). With $W=W_r$ as in \cite[Section 5]{PYY25a}, the construction in
	\cite[Section~8.1]{PYY25a} gives compact sets
	\[
	\Lambda_W^E(\lambda_0),\qquad \Lambda_W^F(\lambda_0)
	\subset \Lambda\setminus W_r
	\]
	of forward and backward infinite hyperbolic times, respectively, both having
	positive $\mu$-measure.  By ergodicity, there is $\tau_*\ge0$ such that
	\[
	H:=f_{\tau_*}\bigl(\Lambda_W^F(\lambda_0)\bigr)
	\cap \Lambda_W^E(\lambda_0)
	\]
	has positive measure; this is the reduction used in the proof of
	\cite[Lemma~8.2]{PYY25a}.  After intersecting $H$
	with a sufficiently small ball and then taking a compact subset, we obtain a
	compact set $K\subset H$ such that
	$
	\mu(K)>0
	$ and $\operatorname{diam}K
	$
	is as small as needed. Also note that $H$ is away from $\Sing(X)\cap\Lambda$.
	
	Following the proof of \cite[Lemma~8.2]{PYY25a}, a single fixed constant
	acceleration of the vector field makes the points of $H$ uniformly admissible
	both as initial forward-hyperbolic endpoints and as terminal
	backward-hyperbolic endpoints in Liao's shadowing lemma.  Translating back to
	the original parametrization changes only the uniform constants.  Therefore,
	a return segment $(x,T)$ with $x,f_T(x)\in K$ and $\operatorname{diam}K$ below the Liao
	closing $\tilde\delta$ scale is a uniformly quasi-hyperbolic finite string with a small
	closing jump.
	
	We apply the finite Liao--Gan shadowing lemma to this string; see
	\cite{Gan} and compare \cite[Lemma 4.5]{PYY21} and \cite[Appendix~B, Lemma~B.1]{PYY25a}.  It produces a hyperbolic
	periodic point $p$ and an increasing time change $\theta$ with
	$\theta(T)=q$ such that the orbit of $p$ (which has period $q$) scaled-shadows $(x,T)$ and
	$|q-T|$ is uniformly bounded.\footnote{\cite{Gan} focuses on a discrete sequence of maps which corresponds to the flow-holonomies $\cP_{1,x_k}$ and therefore does not explicitly contain the estimate for the uniformly bounded time change. See an outline of this fact in \cite[Appendix B]{PYY25a}. } In this setting, $f_{\theta(t)}(p)$ is indeed on the normal plane of $f_t(x)$ for all $t$.
	
	We finally remove the time change.  The time-control argument in the proof
	of \cite[Proposition~7.5]{PYY25a} bounds each flight-time error by the sum
	of a forward and a backward geometric tail.  Summing the same estimate up
	to an arbitrary intermediate time gives
	\[
	\sup_{0\le s\le T}|\theta(s)-s|\le C\rho,
	\]
	where $\rho$ is the scaled-shadowing accuracy and $C$ is independent of
	$x$ and $T$.  Together with the uniform continuity of the flow and the fact
	that $K$ stays a positive distance from the singularities, this implies
	\[
	d_T(x,p)<\epsilon
	\]
	once the closing scale is chosen sufficiently small.  The same estimates give
	$|q-T|\le\kappa$ with $\kappa$ independent of $x$ and $T$. Furthermore, since $\Lambda$ is a chain recurrence class, the periodic orbit is contained in $\Lambda$ due to \cite[Lemma B.1 (f)]{PYY25a}. This completes the proof of the lemma. 
\end{proof}

\begin{proof}[Proof of Proposition \ref{prop.cumulative-lower}]
	Fix $\epsilon>0$ small that Lemma
	\ref{lem.recurrent-hyperbolic-closing-block} applies and, in addition,
	\begin{equation*}\label{eq.Gibbs-scale-choice}
		4\epsilon<\epsilon_{ G},
	\end{equation*}
	where $\epsilon_{ G}>0$ is given by \ref{prop.counting-inputs}. 
	Let $K$ be the compact set given by Lemma
	\ref{lem.recurrent-hyperbolic-closing-block} and put
	\[
	\alpha:=\mu(K)>0.
	\]
	The next step is to show that for each $n\in\NN$ there exists some $T_n$ with $|T_n-n|$ bounded over $n$, and satisfy $\mu(K\cap f_{-T_n}(K)) > 0$.
	
	Since $\mu$ is ergodic for the flow $(f_t)_{t\in\RR}$, there exists $a>0$ for which $F:=f_{a}$ is ergodic w.r.t.\ $\mu$ (see \cite{PS71} and \cite[Theorem 3.3.13]{FH19}).  Since $\mu(K)>0$,  there is $m\ge0$ such that the set
	\[
	U_m:=\bigcup_{j=0}^m F^{-j}K
	\]
	satisfies $\mu(U_m)>1-\frac{\alpha}{2}$.
	For every integer $k\ge0$,
	\begin{align*}
		\sum_{j=0}^m
		\mu\bigl(K\cap F^{-(k+j)}K\bigr)
		\ge \mu\bigl(K\cap F^{-k}U_m\bigr) >\frac{\alpha}{2}.
	\end{align*}
	By the pigeonhole principle, for every $k\ge0$ there is $j(k)\in\{0,\ldots,m\}$ such
	that
	\begin{equation*}\label{eq.uniform-return-mass}
		\mu\bigl(K\cap F^{-(k+j(k))}K\bigr)
		\ge \frac{\alpha}{2(m+1)}:=\beta>0.
	\end{equation*}
	
	Now, given a sufficiently large integer $n$, let
	\[
	k_n:=\left\lceil\frac{n}{a}\right\rceil,
	\qquad
	T_n:=a\bigl(k_n+j(k_n)\bigr).
	\]
	Then
	\begin{equation}\label{eq.return-time-window}
		n\le T_n\le n+a(m+1),
	\end{equation}
	and the compact return set
	\[
	K_n:=K\cap f_{-T_n}K
	\]
	satisfies
	\begin{equation*}
		\mu(K_n)\ge\beta.
	\end{equation*}
%	and then Lemma \ref{lem.upper-Gibbs-K-partition-lower} gives 
%	\begin{equation*}
%		\Lambda(K_n,\epsilon,n)\ge \frac{\beta}{Q'} e^{hn}.
%	\end{equation*}
	
	%	
	%	\medskip
	%	\noindent\emph{A large recurrent separated set.}
	
	%	Maximality gives
	%	\[
	%	K_n\subset
	%	\bigcup_{x\in E_n}\overline{B}_{T_n}(x,4\epsilon).
	%	\]
	%	By \eqref{eq.Gibbs-scale-choice} and \eqref{eq.upper-Gibbs-MME},
	%	\begin{align*}
		%		\beta
		%		&\le \mu(K_n)\\
		%		&\le \sum_{x\in E_n}
		%		\mu\bigl(\overline{B}_{T_n}(x,4\epsilon)\bigr)\\
		%		&\le \sum_{x\in E_n}
		%		\mu\bigl(B_{T_n}(x,\epsilon_{ G})\bigr)\\
		%		&\le Q_{ G}e^{-hT_n}\#E_n.
		%	\end{align*}
	%	Thus
	%	\begin{equation}\label{eq.recurrent-separated-cardinality}
		%		\#E_n\ge c_0e^{hT_n},
		%		\qquad c_0:=\frac{\beta}{Q_{ G}}>0.
		%	\end{equation}
	
	Let $E_n\subset K_n$ be a maximal $(T_n,4\epsilon)$-separated set, and so 
	\begin{align*}
		\beta\le \mu(K_n)\le \sum_{x\in E_n}\mu(\overline B_{T_n}(x,4\epsilon))\le \sum_{x\in E_n}\mu(B_{n}(x,\epsilon_G))\le Q'e^{-hn}\# E_n
	\end{align*}
	where the last inequality follows from Proposition \ref{prop.counting-inputs}. This gives 
	$$
	\# E_n\ge \frac{\beta}{Q'}e^{hn}.
	$$
	
	For every $x\in E_n$, we have $x,f_{T_n}(x)\in K$.  Lemma
	\ref{lem.recurrent-hyperbolic-closing-block} therefore gives a periodic
	point $p_x\in\Lambda$ and a return time $q_x>0$ such that
	\begin{equation}\label{eq.recurrent-periodic-closing}
		f_{q_x}(p_x)=p_x,
		\qquad |q_x-T_n|\le\kappa,
		\qquad d_{T_n}(x,p_x)<\epsilon.
	\end{equation}
	Define the set 
	\[
	P_n:=\{p_x:x\in E_n\}.
	\]
	If $x,y\in E_n$ are distinct, then
	$d_{T_n}(x,y)\ge4\epsilon$, and hence
	\[
	d_{T_n}(p_x,p_y)
	\ge d_{T_n}(x,y)-d_{T_n}(x,p_x)-d_{T_n}(y,p_y)
	>2\epsilon.
	\]
	Thus $P_n$ is $(T_n,2\epsilon)$-separated.  In particular, the map
	$x\mapsto p_x$ is injective and
	\begin{equation}\label{eq.recurrent-periodic-points-cardinality}
		\#P_n=\#E_n\ge \frac{\beta}{Q'}e^{hn}.
	\end{equation}
	
%	\medskip
%	\noindent\emph{From periodic points to primitive periodic orbits.}
	Next we deal with the possibility that multiple $p_x$ may lie on the same orbit. The argument is standard. By uniform continuity and the compactness of $\Lambda$, there is
	$b_0>0$ such that
	\begin{equation*}\label{eq.recurrent-small-phase}
		d(f_r(z),z)<2\epsilon
		\qquad
		\forall z\in\Lambda\text{ and }|r|<b_0.
	\end{equation*}
	Suppose that $p,p'\in P_n$ lie on the same primitive periodic orbit
	$\gamma$ within time $r$ of each other.  If $|r|<b_0$, then for every
	$s\in[0,T_n]$,
	\[
	d(f_s(p),f_s(p'))
	=d\bigl(f_s(p),f_r(f_s(p))\bigr)<2\epsilon,
	\]
	contradicting the $(T_n,2\epsilon)$-separation of $P_n$.  Hence $|r|\ge  b_0$ for all pairs of points $p,p'\in P_n$, and therefore for every periodic orbit $\gamma,$ it must hold that 
	\begin{equation}\label{eq.recurrent-phase-multiplicity}
		\#(P_n\cap\gamma)
		\le 1+\frac{\ell(\gamma)}{b_0}
	\end{equation}
	where $\ell(\gamma)$ is the prime period of $\gamma$ (not to be confused with its length).
	
	Let $\gamma_x$ be the primitive periodic orbit containing $p_x$.  Since
	$q_x$ is a period of $p_x$, by  \eqref{eq.recurrent-periodic-closing} and \eqref{eq.return-time-window} we have
	\[
	\ell(\gamma_x)\le q_x
	\le T_n+\kappa
	\le n+a(m+1)+\kappa.
	\]
	To simplify notation, set
	\[
	\tau:=a(m+1)+\kappa.
	\]
	which does not depend on $n.$
	It follows from \eqref{eq.recurrent-phase-multiplicity} that, after
	increasing a constant $C_0>0$ if necessary,
	\[
	\#(P_n\cap\gamma)\le C_0n
	\]
	for every primitive orbit $\gamma$ meeting $P_n$ and every sufficiently
	large $n$.  Combining this with \eqref{eq.recurrent-periodic-points-cardinality}, we obtain
	\begin{equation}\label{eq.recurrent-integer-lower}
		\#\Pi_\Lambda(n+\tau)
		\ge \frac{\#P_n}{C_0n}
		\ge \frac{\beta}{Q'C_0}\frac{e^{hn}}{n}.
	\end{equation}
	%Notice that no lower estimate for $q_x$ is used, and $q_x$ need not be the prime period of $p_x$.
	
	Finally, for sufficiently large real $t$, let $n:=\lfloor t-\tau\rfloor.$ By monotonicity of $\#\Pi_\Lambda$ w.r.t.\ $t$ and
	\eqref{eq.recurrent-integer-lower}, we have
	\[
	\#\Pi_\Lambda(t)
	\ge \#\Pi_\Lambda(n+\tau)
	\ge \frac{\beta}{Q'C_0}\frac{e^{hn}}{n}
	\ge \frac{\beta e^{-h(\tau+1)}}{Q'C_0}\frac{e^{ht}}{t}.
	\]
	This proves the proposition.
\end{proof}

\subsection{The upper bound}\label{ss.6.2}
The goal of this section is to establish the upper bound of $\#\Pi_\Lambda(t)$ of the form $const\cdot{e^{ht}}/{t}$. See Propositions \ref{prop.narrow-window-upper} and \ref{prop.cumulative-upper} below.

Note that the previous lower bound estimate does not require Komuro expansivity. 
The next lemma is the point at which it is used. It is motivated by \cite{Kni99}: instead of
choosing one representative on each periodic orbit $\gamma$, we choose a number of
points proportional to its period $\ell(\gamma)$. 

\begin{lemma}\label{lem.phase-separated}
	There exist constants $b>0$, $\delta_{\mathrm{per}}>0$, and
	$\epsilon_{\mathrm{per}}>0$ such that the following holds. For every
	sufficiently large $t$ and every periodic orbit
	$\gamma\in\Pi_\Lambda(t,\delta_{\mathrm{per}})$
	choose any $x_\gamma\in\gamma$ and put
	\[
	m_\gamma=\left\lfloor\frac{\ell(\gamma)}b\right\rfloor,
	\qquad
	E_\gamma=\{f_{jb}(x_\gamma):0\le j<m_\gamma\}.
	\]
	Then the set
	\[
	E_t:=\bigcup_{\gamma\in \Pi_\Lambda(t,\delta_{\mathrm{per}})}E_\gamma
	\]
	is $(t,\epsilon_{\mathrm{per}})$-separated.
\end{lemma}

\begin{proof}
	By Theorem \ref{m.e}, $X|_\Lambda$ is Komuro expansive. For any $b_1>0$, let $\vep_K>0$ be given by the definition of Komuro expansivity.  
	For any $b>4b_1$, by uniform continuity we can choose $\delta_{\mathrm{per}}>0$ and
	$\epsilon_{\mathrm{per}}>0$ such that
	\begin{equation}\label{eq.Komuro-scales}
		\epsilon_{\mathrm{per}}+
		\sup_{\substack{z\in\Lambda\\|s|\le\delta_{\mathrm{per}}}}
		d(f_s(z),z)<\vep_K.
	\end{equation}

	Towards a contradiction, suppose that two distinct points $x\in E_\gamma$ and $y\in E_{\gamma'}$ are not
	$(t,\epsilon_{\mathrm{per}})$-separated (we do not assume $\gamma\ne\gamma'$). Then
	\[
	d(f_s(x),f_s(y))<\epsilon_{\mathrm{per}}
	\quad\text{for every }s\in[0,t].
	\]
	Write
	$p=\ell(\gamma)$ and $q=\ell(\gamma')$ and note that they both belong to $(t-\delta_{\mathrm{per}},t]$ by assumption. 
	We define a time reparametrization $\theta:\mathbb R\to\mathbb
	R$ by
	\[
	\theta(kp+s)=kq+\frac qp s,
	\qquad k\in\mathbb Z,\quad 0\le s<p.
	\]
	For $r=kp+s$, periodicity and \eqref{eq.Komuro-scales} give
	\begin{align*}
		d(f_r(x),f_{\theta(r)}(y))
		&=d\left(f_s(x),f_{(q/p)s}(y)\right)\\
		&\le d(f_s(x),f_s(y))
		+d\left(f_s(y),f_{(q/p)s}(y)\right)\\
		&<\vep_K,
	\end{align*}
	where we used $s\le p\le t$ and
	$|(q/p)s-s|\le|q-p|\le\delta_{\mathrm{per}}$. %Hence the two full orbits 	$\vep_K$-shadow one another after the time change $\theta$. 
	Komuro
	expansivity then implies that $x$ and $y$ lie on the same orbit, that is,
	$\gamma=\gamma'$ and $p=q$. In this case $\theta$ is the
	identity map. Furthermore, there exists $t_0\in\RR$ for which 
	$$
	f_{\theta(t_0)}(y)\in f_{[t_0-b_1,t_0+b_1]}(x),
	$$
	Since $\theta$ is the identity, the (circular) time difference between $x,y$ is at most $a$.  
	This contradicts the assumption that the  time difference between $x$ and $y$ is at least $b > 4a$. Therefore $x=y$, proving that $E_t$ is $(t,\epsilon_{\mathrm{per}})$-separated.
\end{proof}

\begin{remark}\label{r.6.5}
	Note that $b_1>0$ does not need to be small. One can also take $\epsilon_{\mathrm{per}} = \vep_K/2$ (or even arbitrarily close to $\vep_K$) then choose $\delta_{\mathrm{per}}$ accordingly. 
\end{remark}

Next we state the precise upper bound estimate for $\#\Pi_\Lambda(t,\delta_{\mathrm{per}})$
\begin{proposition}\label{prop.narrow-window-upper}
	Under the assumptions of Theorem \ref{m.counting}, there is $C_1>0$ such that, for all sufficiently large $t$,
	\begin{equation}\label{eq.narrow-window-upper}
		\#\Pi_\Lambda(t,\delta_{\mathrm{per}})
		\le C_1\frac{e^{ht}}t.
	\end{equation}
\end{proposition}

\begin{proof} 
	We note that the scale $\vep_G$ in Proposition \ref{prop.counting-inputs} can be chosen arbitrarily small, as we previously explained in a footnote for Proposition \ref{prop.counting-inputs}. In particular one can take $\vep_G = \vep/(100L_X)$ according to \cite[Section 6.2]{PYY25a} and \cite{PYYY}, where $\vep>0$ is the two-scale constants in the improved Climenhaga-Thompson criterion. In \cite[Section 6.2, above Equation (6.5)]{PYY25a} $\vep$ can be taken arbitrarily small (smaller than $\vep_K$, for instance), and the specification scale $\delta$ is taken after that. From Remark \ref{r.6.5} we get that $\epsilon_{\mathrm{per}} = \vep_K/2 > \epsilon_G$. Then, the previous lemma together with Proposition \ref{prop.counting-inputs} implies that
	\[
	\sum_{\gamma \in \Pi_\Lambda(t,\delta_{\mathrm{per}})}
	\left\lfloor\frac{\ell(\gamma)}b\right\rfloor
	=\#E_t \le \Lambda(\Lambda, \epsilon_{\mathrm{per}}, t) \le \Lambda(\Lambda, \epsilon_{G}, t)
	\le Qe^{ht}.
	\]
	Here we increase $Q$ in Proposition \ref{prop.counting-inputs} such that the upper estimate in \eqref{eq.global-partition-upper} extends from integer times to all real times.

	For large $t$, every summand on the left is at least $t/(2b)$. Therefore
	\[
	\frac{t}{2b}\#\Pi_\Lambda(t,\delta_{\mathrm{per}})
	\le Qe^{ht},
	\]
	which proves \eqref{eq.narrow-window-upper}.
\end{proof}

This leads to the following cumulative bound. 
\begin{proposition}\label{prop.cumulative-upper}
	Under the assumptions of Theorem \ref{m.counting}, there is $C_2>0$ such that
	\[
	\#\Pi_\Lambda(t)\le C_2\frac{e^{ht}}t
	\]
	for every sufficiently large $t$.
\end{proposition}

\begin{proof}
	Partition the interval of possible periods into intervals of width
	$\delta_{\mathrm{per}}$. Proposition
	\ref{prop.narrow-window-upper} then gives
	\[
	\#\Pi_\Lambda(t)
	\le \#\Pi_\Lambda(T_0)+
	C_1\sum_{k\ge0:\,t-k\delta_{\mathrm{per}}\ge T_0}
	\frac{e^{h(t-k\delta_{\mathrm{per}})}}
	{t-k\delta_{\mathrm{per}}}.
	\]
	For the terms with $t-k\delta_{\mathrm{per}}\ge t/2$, the denominator is
	at least $t/2$ and consequently the resulting geometric series is bounded by a constant
	multiple of $e^{ht}/t$. The remaining terms are bounded by a constant
	multiple of $e^{ht/2}$, which is less than $O(e^{ht}/t)$. This proves the
	claim.
\end{proof}

\subsection{A fixed period window}

Combining Proposition \ref{prop.cumulative-lower} and \ref{prop.cumulative-upper} we see that, for some $C>1$ large, it holds 
$$
C^{-1}\frac{e^{ht}}{t}\le \#\Pi_\Lambda(t)\le C\frac{e^{ht}}{t}.
$$
This is precisely \eqref{eq.cumulative-counting}.

To obtain \eqref{eq.window-counting}, choose $\Delta>0$ so large that
\[
2Ce^{-h\Delta}<\frac{C^{-1}}{2}.
\]
Then, for $t$ sufficiently large,
\begin{align*}
	\#\Pi_\Lambda(t,\Delta)
	&=\#\Pi_\Lambda(t)-\#\Pi_\Lambda(t-\Delta)\\
	&\ge C^{-1}\frac{e^{ht}}t
	-C\frac{e^{h(t-\Delta)}}{t-\Delta}\\
	&\ge \frac{C^{-1}}{2}\frac{e^{ht}}t.
\end{align*}
The upper bound of $\#\Pi_\Lambda(t,\Delta)$ follows immediately from
$\#\Pi_\Lambda(t,\Delta)\le\#\Pi_\Lambda(t)$ and Proposition
\ref{prop.cumulative-upper}. This finishes the proof of \eqref{eq.window-counting}.

\subsection{Equidistribution}
We use the standard entropy argument for separated periodic orbits; see
\cite[Section~2.3]{BCFT}. We first record the narrow-window version. Here $\delta_{\mathrm{per}} >0$ is given by Lemma \ref{lem.phase-separated}.

\begin{lemma}\label{l.narrow}
Let $(s_k)_k$ be a sequence of real numbers tending to infinity, and for each $k$, let
$
\cA_k\subset \Pi_\Lambda(s_k,\delta_{\mathrm{per}})
$
be a collection of periodic orbits with
\begin{equation}\label{eq.full-entropy-narrow-family}
	\lim_{k\to\infty}
	\frac{1}{s_k}\log\#\cA_k=h.
\end{equation}
Then
\begin{equation*}
	\frac{1}{\#\cA_k}
	\sum_{\gamma\in\cA_k}\mu_\gamma
	\xrightarrow[k\to\infty]{w^*}
	\mu_{\mathrm{MME}}.
\end{equation*}
\end{lemma}

\begin{proof}
	Choose one point $x_\gamma\in\gamma$ for every
	$\gamma\in\cA_k$ to form a set $A_k$.
	By Lemma \ref{lem.phase-separated}, $A_k$ is a subset of $E_{s_k}$ and is therefore
	$(s_k,\epsilon_{\mathrm{per}})$-separated. Consider
	$$
	\eta_k
	:=
	\frac{1}{\#A_k}
	\sum_{x\in A_k}
	\frac{1}{s_k}\int_0^{s_k}\delta_{f_r(x)}\,dr.
	$$
	By \eqref{eq.full-entropy-narrow-family} and the standard entropy estimate
	for empirical measures of separated sets, every weak-* accumulation point
	$\eta$ of $(\eta_k)$ satisfies
	$$
	h_\eta(X|_\Lambda)\ge h.
	$$
	Since $h=h_{\mathrm{\mathrm{top}}}(X|\Lambda)$, it follows that $\eta$ is a measure of maximal entropy and, by
	uniqueness, must coincide with $\mu_{\mathrm{MME}}$.

	On the other hand, since $|s_k-\ell(\gamma)|<\delta_{\mathrm{per}}$ for all $\gamma\in\cA_k,$ we see that 
	$$
	\frac{1}{\#\cA_k}
	\sum_{\gamma\in\cA_k}\mu_\gamma
	$$
	must converge to the same weak-* limit as $\eta_k$. This proves the lemma.
\end{proof}

We now pass from narrow windows to the fixed window of size $\Delta$. 

\begin{lemma}\label{l.seq}
	For every sequence of real numbers $(t_k)_k$ tending to infinity, there is a subsequence $t_{n_k}$ along which $\nu_{t_{n_k},\Delta}$ converges to $\mu_{\mathrm{MME}}$ in weak-*. 
\end{lemma}
\begin{proof}
	
	We may assume w.l.o.g.\ that ${\Delta}/{\delta_{\mathrm{per}}}$ is an integer (by slightly altering $\Delta$), which we call $N$.
	For $1\le j\le N$, define
	$$
	s_j(t):=t-(j-1)\delta_{\mathrm{per}},
	\qquad
	\cA_j(t):=\Pi_\Lambda(s_j(t),\delta_{\mathrm{per}}).
	$$
	Then we have 
	$$
	\Pi_\Lambda(t,\Delta)=\bigsqcup_{j=1}^N\cA_j(t).
	$$
	To simplify notation, write
	$$
	n_j(t):=\#\cA_j(t),
	\qquad
	w_j(t):=\frac{n_j(t)}{\#\Pi_\Lambda(t,\Delta)}.
	$$
	Whenever $n_j(t)>0$, put
	$$
	\tilde \nu_{j,t}:=
	\frac{1}{n_j(t)}
	\sum_{\gamma\in\cA_j(t)}\mu_\gamma.
	$$
	Then direct calculation yields
	\begin{equation}\label{eq.fixed-window-convex-decomposition}
		\nu_{t,\Delta}
		=
		\sum_{{1\le j\le N}, n_j(t)>0}
		w_j(t)\tilde\nu_{j,t}.
	\end{equation}
	
	Now fix any sequence $t_k\to\infty$ and note that $N$ is fixed. Then, there is a subsequence, still denoted by $t_k$, such that each  
	$
	w_j(t_k)$ converges to some $w_j$ as $k\to\infty$, $1\le j\le N.$ The terms for which $w_j=0$ have vanishing total weight.
	Suppose that $w_j>0$, then for all sufficiently large $k$,
	$$
	n_j(t_k)
	\ge
	\frac{w_j}{2}\#\Pi_\Lambda(t_k,\Delta)
	\ge
	c_j\frac{e^{ht_k}}{t_k}.
	$$
	This, together with the upper bound given in Proposition \ref{prop.narrow-window-upper}, yields (note that $|t_k-s_j(t_k)|\le\Delta$)
	$$
	\lim_{k\to\infty}
	\frac{1}{s_j(t_k)}\log n_j(t_k)=h.
	$$
	Lemma \ref{l.narrow} therefore gives
	$$
	\tilde\nu_{j,t_k}\xrightarrow[k\to\infty]{w^*}\mu_{\mathrm{MME}}
	$$
	whenever $w_j>0$.
	It now follows
	from \eqref{eq.fixed-window-convex-decomposition} that 
	$$
	\nu_{t_k,\Delta}\to\mu_{\mathrm{MME}}
	$$
	in the weak-* topology.
\end{proof}

\begin{proof}[Proof of \eqref{eq.window-equidistribution}: convergence of $\nu_{t,\Delta}$ to $\mu_{\mathrm{MME}}$]
	Assume that this is not true, then by the compactness of the set of all probability measures on $M$ under the weak-* topology, one can find a sequence of time $t_k\to\infty$ along which $\nu_{t_k,\Delta}$ converges to another measure $\nu$. However, Lemma \ref{l.seq} shows that there is a subsequence $t_{n_k}$ along which $\nu_{t_{n_k},\Delta}$ converges to $\mu_{\mathrm{MME}}$, which is a contradiction.
\end{proof}

Finally, we prove cumulative equidistribution. Fix $K\ge1$. The interval
$(0,t]$ is the disjoint union of
$$
(0,t-K\Delta]
\quad\text{and}\quad
(t-(j+1)\Delta,t-j\Delta],
\qquad 0\le j<K.
$$
Set
$$
\overline  r_{t,K}
:=
\frac{\#\Pi_\Lambda(t-K\Delta)}{\#\Pi_\Lambda(t)}, \qquad  r_{j,t} = \frac{\#\Pi_\Lambda(t-j\Delta,\Delta)}{\#\Pi_\Lambda(t)},\qquad 0\le j<K.
$$
Then we have the convex decomposition
\begin{equation*}
	\overline \nu_{t} =\overline r_{t,K} \cdot \overline \nu_{t-K\Delta} + \sum_{j=0}^{K-1} r_{j,t} \cdot \nu_{t-j\Delta,\Delta}.
\end{equation*}
Consequently, for any continuous function $\phi,$ we have 
\begin{align*}
	\left|
	\bar\nu_t(\phi)-\mu_{\mathrm{MME}}(\phi)
	\right|
	&\le
	\overline r_{t,K}
	\left|
	\bar\nu_{t-K\Delta}(\phi)-\mu_{\mathrm{MME}}(\phi)
	\right|\\
	&\quad+
	\sum_{j=0}^{K-1}r_{j,t}
	\left|
	\nu_{t-j\Delta,\Delta}(\phi)
	-\mu_{\mathrm{MME}}(\phi)
	\right|\\\numberthis\label{e.last}
	&\le
	2|\phi|_{C^0}\cdot \overline r_{t,K}
	+
	\max_{0\le j<K}
	\left|
	\nu_{t-j\Delta,\Delta}(\phi)
	-\mu_{\mathrm{MME}}(\phi)
	\right|.
\end{align*}

By the two-sided cumulative estimate \eqref{eq.cumulative-counting}, the error term satisfies
$$
	\overline r_{t,K}
	\le
	C^2 e^{-hK\Delta}\frac{t}{t-K\Delta}
$$
whenever $t-K\Delta$ is sufficiently large. In particular,
\begin{equation}\label{e.r}
	\limsup_{t\to\infty}\overline r_{t,K}
	\le C^2e^{-hK\Delta}.
\end{equation}
Note that as $t\to\infty$ with $K$ fixed, each $\nu_{t-j\Delta,\Delta}$ converges to $\mu_{\mathrm{MME}}$ in weak-* topology, and therefore the max term vanishes. Thus we obtain from \eqref{e.last} and \eqref{e.r} that
$$
\limsup_{t\to\infty}\left|\overline \nu_{t}(\phi) - \mu_{\mathrm{MME}}(\phi)\right|\le 2C^2|\phi|_{C^0} e^{-hK\Delta}.
$$
Sending $K$ to infinity then gives $\overline \nu_{t}\to \mu_{\mathrm{MME}}$ in weak-* as desired. This concludes the proof of Theorem \ref{m.counting}.

\begin{proof}[Proof of Theorem \ref{m.robust.counting}]
	We note that in the proof of Theorem \ref{m.counting} above, the following main ingredients are used:
	\begin{itemize}
		\item A Gibbs upper bound and the existence of the MME are used to obtain the lower bound on $\#\Pi_\Lambda(t)$.
		\item A partition sum upper bound and Komuro expansivity show the upper bound on $\#\Pi_\Lambda(t,\delta_{\mathrm{per}})$.
		\item Uniqueness of the MME on $\Lambda$ leads to equidistribution. 
	\end{itemize}
	All ingredients have been established on $\Lambda_Y$ for vector fields $Y$ that are $C^1$ close to $X$: the existence and uniqueness of the MME can be found in \cite[Thereom H]{PYY25a} under the exact same assumptions; Gibbs upper bound and the partition sum upper bound can be found in \cite[Section 9]{PYY25a} where the assumptions of \cite{PYYY} are verified; finally, Komuro expansivity is proven by Theorem \ref{m.e}. This establishes Theorem \ref{m.robust.counting} and concludes this section. 
\end{proof}

\section{Application: star flows}

In this section we apply Theorem \ref{m.e}, \ref{m.counting} and \ref{m.robust.counting} to star flows. Recall that a star flow is a flow such that all nearby flows have only hyperbolic critical elements. The main result of \cite{BdL} shows that on a $C^1$ open and dense subset, the chain recurrent set of any star flow can be decomposed into finitely many multi-singular hyperbolic pieces. Even though the original definition of multi-singular hyperbolicity in \cite{BdL} is different from the one in Definition \ref{def.multi-sing-hyp}, it is proven in \cite{CLYZ} that on a $C^1$ open and dense subset, these two definitions coincide.

\subsection{Proof of Theorem \ref{m.e.star}}
\begin{proof}
	By \cite[Theorem~3]{BdL} and \cite{CLYZ} there exists a $C^1$ open and dense subset
	$\cO\subset\xX^1_*(M)$ such that, for every $X\in\cO$,
	the chain recurrent set $\CR(X)$ is contained in the union of finitely
	many pairwise disjoint filtrating regions $R_1,\ldots,R_m,$
	and $X$ is multi-singular hyperbolic in each $R_i$ in the sense of Definition \ref{def.multi-sing-hyp}. We remark that $\CR(X)\cap R_j$ may not be chain transitive for each $j$.
	
	Fix $X\in\cO$, and for $1\le i\le m$ let
	\[
	\Lambda_i:=\bigcap_{t\in\RR}f_t^X(R_i)
	\]
	be the maximal invariant set in $R_i$. Then we have 
	\begin{equation*}
		\CR(X)\subset\bigcup_{i=1}^m\Lambda_i.
	\end{equation*}
	Moreover, each $\Lambda_i$ is a multi-singular hyperbolic set. %If one of 	the sets $\Lambda_i$ consists only of singularities, then it is a finite 	set of isolated hyperbolic singularities, and the robust expansivity 	conclusion below is immediate for this set. Below we only consider those $\Lambda_i$ that are non-trivial.
	
	Apply Theorem \ref{m.e} to each $\Lambda_i$. We obtain a $C^1$
	neighborhood $\cV_i$ of $X$ and an open neighborhood $U_i$ of
	$\Lambda_i$ such that, for every $Y\in\cV_i$, the restriction of the
	flow of $Y$ to
	\[
	\Lambda_i(Y):=\bigcap_{t\in\RR}f_t^Y(U_i)
	\]
	is Komuro expansive. Furthermore, for each fixed $\vep>0$, the
	corresponding Komuro expansivity constant can be chosen uniformly over
	all $Y\in\cV_i$. We may assume that all $U_i$ have disjoint closure. 
	The set
	\[
	U:=\bigcup_{i=1}^m U_i
	\]
	is an open neighborhood of $\CR(X)$. Since the chain recurrent set varies upper semicontinuously with the
	vector field, after shrinking a $C^1$ neighborhood $\cV_0$ of $X$, we
	have
	\[
	\CR(Y)\subset U
	\qquad\text{for every }Y\in\cV_0.
	\]
	Since $\CR(Y)$ is invariant and $U_i$ are pairwise disjoint, it
	follows that $\CR(Y)$ decomposes into the finitely many $\Lambda_i(Y)\subset U_i$.
	Below we show that every $Y\in\cV$ is Komuro expansive when restricted to its chain recurrent set.
	
	Fix $\vep>0$. For each $i$, let $\delta_i>0$ be a Komuro expansivity
	constant for the maximal invariant set $\Lambda_i(Y)$, chosen uniformly
	over $Y\in\cV$. Since the closures of the $U_i$ are pairwise disjoint,
	\[
	d_0:=
	\min_{i\ne j}d(\overline{U_i},\overline{U_j})>0.
	\]
	Define
	\[
	\delta:=
	\min\left\{\frac{d_0}{2},\delta_1,\ldots,\delta_m\right\}>0.
	\]
	Let $Y\in\cV$, let $x,y\in\CR(Y)$, and let
	$\theta:\RR\to\RR$ be a time reparametrization with
	$\theta(0)=0$ such that
	\[
	d\bigl(f_t^Y(x),f_{\theta(t)}^Y(y)\bigr)<\delta
	\qquad\text{for every }t\in\RR.
	\]
	Since $\delta<d_0/2$, $x,y$ must be contained in the same $\Lambda_i$. 
	The uniform Komuro expansivity of $\Lambda_i(Y)$ now gives a time
	$t_0\in\RR$ such that
	\[
	f_{\theta(t_0)}^Y(y)
	\in f_{[t_0-\vep,t_0+\vep]}^Y(x).
	\]
	The constant $\delta$ depends on $\vep$ and on the fixed neighborhood
	$\cV$, but not on $Y\in\cV$. This proves the theorem.
\end{proof}

\subsection{Proof of Theorem \ref{m.counting.star}}
\begin{proof}[Proof of Theorem\ref{m.counting.star}]
	By \cite[Theorem C]{PYY25a}, there exists a $C^1$ open and dense set $\cO_1\subset\xX^1_*(M)$ such that every $X\in \cO_1$ with topological entropy $h>0$ has a unique measure of maximal entropy $\mu$. Furthermore, there exists an open set $U\subset M$ with $\CR(X)\cap\partial U=\emptyset$ and $\eta>0$ satisfying the following entropy separation properties:
	\[
	\Lambda:=\bigcap_{t\in\RR} f_t(U),
	\qquad
	K:=\CR(X)\setminus U
	\]
	are compact invariant sets; moreover, $\Lambda$ carries the unique
	measure of maximal entropy of $X$, and
	\begin{equation}\label{eq.star-gap-fixed-X}
		h_{\mathrm{\mathrm{top}}}(X|_{\Lambda})
		=
		h_{\mathrm{\mathrm{top}}}(X)=h,
		\qquad
		h_{\mathrm{\mathrm{top}}}(X|_{K})
		\le h-\eta.
	\end{equation}
	
	Moreover, by the robust thermodynamic estimates established in
	\cite{PYY25a}, the conclusions of Theorem \ref{m.robust.counting} hold
	for the maximal invariant set in $U$. Next let $\cO_{\mathrm{K}}\subset\xX^1_*(M)$ be the open and dense set in Theorem 	\ref{m.e.star}. Then
	\[
	\cO:=\cO_{1}\cap\cO_{\mathrm{K}}
	\]
	is open and dense in $\xX^1_*(M)$. For each $X\in\cO$,
	by Theorem \ref{m.robust.counting}, there exist
	$\Delta>0$, $C_0>1$, and $t_0>0$ such that, for the maximal invariant set $\Lambda\subset U,$
	\begin{equation}\label{eq.star-local-counting}
		C_0^{-1}\frac{e^{ht}}t
		\le \#\Pi_\Lambda(t,\Delta)
		\le C_0\frac{e^{ht}}t,
		\qquad
		C_0^{-1}\frac{e^{ht}}t
		\le \#\Pi_\Lambda(t)
		\le C_0\frac{e^{ht}}t
	\end{equation}
	for all $t\ge t_0$, and the corresponding fixed-window and cumulative
	periodic-orbit averages on $\Lambda$ converge to
	$\mu_{\mathrm{MME}}$.

	We claim that the periodic orbits contained in $K = \CR(X)\setminus U$ have exponentially
	smaller growth. More precisely, with
	\[
	a:=h-\eta/2,
	\]
	there is $C_1>0$ such that
	\begin{equation}\label{eq.star-remainder-counting}
		\#\Pi_K(t)\le C_1\frac{e^{at}}t
	\end{equation}
	for all sufficiently large $t$.
	
	Indeed, by Theorem \ref{m.e.star}, the flow is Komuro expansive on
	$\CR(X)$. The argument of Lemma \ref{lem.phase-separated} shows the existence of
	$\rho>0$ and $\epsilon_{\mathrm{per}}>0$ such that one representative
	from each orbit in $\Pi_K(t,\rho)$ forms a
	$(t,\epsilon_{\mathrm{per}})$-separated set.
	
	Suppose that the corresponding narrow-window estimate
	\begin{equation}\label{eq.star-remainder-narrow}
		\#\Pi_K(t,\rho)\le C\frac{e^{at}}t
	\end{equation}
	fails for every $C>0$ for some large $t$. Then there exist $t_k\to\infty$ such that
	\[
	\#\Pi_K(t_k,\rho)
	\ge k\frac{e^{at_k}}{t_k}.
	\]
	Choose one point on each orbit in $\Pi_K(t_k,\rho)$ and denote the
	resulting separated set by $E_k$. As in the proof of equidistribution in Theorem \ref{m.counting}, consider the sequence of measures
	\[
	\zeta_k
	:=
	\frac1{\#E_k}
	\sum_{x\in E_k}
	\frac1{t_k}\int_0^{t_k}\delta_{f_s(x)}\,ds.
	\]
	Passing to a subsequence, assume that
	$\zeta_k\to\zeta$ in the weak-* topology. The measure $\zeta$ is
	invariant and supported on $K$. The standard entropy estimate for
	empirical measures of separated sets gives
	\[
	h_\zeta(f_1)
	\ge
	\limsup_{k\to\infty}
	\frac1{t_k}\log\#E_k
	\ge a.
	\]
	This contradicts \eqref{eq.star-gap-fixed-X}, since
	$a=h-\eta/2>h-\eta$. Thus \eqref{eq.star-remainder-narrow} holds for some
	$C>0$. Summing over consecutive windows of width $\rho$, exactly as in
	the proof of Proposition \ref{prop.cumulative-upper}, yields
	\eqref{eq.star-remainder-counting}.
	
	We now combine the two estimates. Since
	\[
	\Pi(t)=\Pi_\Lambda(t)\sqcup\Pi_K(t),
	\]
	the cumulative estimates in \eqref{eq.star-local-counting} and
	\eqref{eq.star-remainder-counting}, together with $a<h$, give
	\[
	C^{-1}\frac{e^{ht}}t
	\le \#\Pi(t)
	\le C\frac{e^{ht}}t
	\]
	for all sufficiently large $t$, after increasing $C$. The global fixed-window
	bounds can be obtained in the same way.
	
	It remains to prove equidistribution. Write
	\[
	\overline \nu^\Lambda_{t}
	:=
	\frac1{\#\Pi_\Lambda(t)}
	\sum_{\gamma\in\Pi_\Lambda(t)}\mu_\gamma.
	\]
	Then Theorem \ref{m.robust.counting} gives 
	\[
	\overline \nu^\Lambda_{t}
	\xrightarrow[t\to\infty]{w^*}\mu_{\mathrm{MME}}.
	\]
	
	Moreover,
	\[
	\frac{\#\Pi_K(t)}{\#\Pi(t)}
	=
	O\bigl(e^{-(h-a)t}\bigr)
	\longrightarrow0.
	\]
	Hence the contribution of the periodic orbits in $K$ to
	$\nu_{t}$ tends to zero, and therefore
	\[
	\overline \nu_{t}
	\xrightarrow[t\to\infty]{w^*}\mu_{\mathrm{MME}}.
	\]
	The same argument gives the equidistribution in the fixed-window case. 
	This completes the proof.
\end{proof}

\end{document}